\documentclass[letterpaper]{amsart}

\usepackage[utf8]{inputenc}
\usepackage{amsmath}
\usepackage{amssymb}
\usepackage{amsthm}
\usepackage{tikz-cd}
\usepackage{mathrsfs}
\usepackage{enumitem}
\usepackage[alphabetic]{amsrefs}
\usepackage{tikz}
\usepackage{float}
\usepackage{hyperref}
\usepackage{multirow}

\theoremstyle{plain}
\newtheorem{teo}{Theorem}[section]
\newtheorem{prop}[teo]{Proposition}
\newtheorem{lemma}[teo]{Lemma}
\newtheorem{cor}[teo]{Corollary}
\theoremstyle{definition}

\newtheorem{defin}[teo]{Definition}
\newtheorem{condition}[teo]{Condition}
\newtheorem{example}[teo]{Example}

\theoremstyle{remark}
\newtheorem{remark}[teo]{Remark}
\newtheorem{claim}[teo]{Claim}
\numberwithin{equation}{section}

\newcommand{\A}{\mathcal{A}}
\newcommand{\Bcat}{\mathcal{B}}
\newcommand{\C}{\mathbb{C}}
\newcommand{\Cheart}{\mathcal{C}}
\newcommand{\D}{\mathcal{D}}
\newcommand{\DC}{\mathrm{D}}
\newcommand{\E}{\mathcal{E}}
\newcommand{\F}{\mathcal{F}}
\newcommand{\Ftors}{\mathscr{F}}
\newcommand{\G}{\mathcal{G}}
\renewcommand{\H}{\mathcal{H}}
\newcommand{\M}{\mathfrak{M}}
\newcommand{\m}{\mathfrak{m}}
\newcommand{\N}{\mathcal{N}}
\renewcommand{\O}{\mathscr{O}}
\renewcommand{\P}{\mathbb{P}}
\newcommand{\Pfam}{\mathcal{P}}
\newcommand{\Pscr}{\mathscr{P}}
\newcommand{\Qscr}{\mathscr{Q}}
\newcommand{\R}{\mathbb{R}}
\newcommand{\Rcal}{\mathcal{R}}
\newcommand{\Scal}{\mathcal{S}}
\newcommand{\T}{\mathcal{T}}
\newcommand{\Ttors}{\mathscr{T}}
\newcommand{\U}{\mathcal{U}}
\newcommand{\W}{\mathcal{W}}
\newcommand{\Z}{\mathbb{Z}}

\let\Im\relax
\let\Re\relax
\DeclareMathOperator{\Ann}{Ann}
\DeclareMathOperator{\Aut}{Aut}
\DeclareMathOperator{\Coh}{Coh}
\DeclareMathOperator{\coker}{coker}
\DeclareMathOperator{\Def}{Def}
\DeclareMathOperator{\End}{End}
\DeclareMathOperator{\Ext}{Ext}
\DeclareMathOperator{\Gr}{Gr}
\DeclareMathOperator{\Hom}{Hom}
\DeclareMathOperator{\Im}{Im} 
\DeclareMathOperator{\im}{im}
\DeclareMathOperator{\Re}{Re}
\DeclareMathOperator{\rk}{rk}
\DeclareMathOperator{\Sing}{Sing}
\DeclareMathOperator{\Spec}{Spec}
\DeclareMathOperator{\Stab}{Stab}

\newcommand{\dashperf}{\mathrm{-perf}}
\newcommand{\HOM}{\mathcal{H}\kern -.5pt om}
\newcommand{\id}{\mathrm{id}}
\newcommand{\num}{\mathrm{num}}
\newcommand{\pr}{\mathrm{pr}}

\newcommand{\abs}[1]{\left\lvert#1\right\rvert}
\newcommand{\norm}[1]{\left\lVert#1\right\rVert}

\begin{document}

\title{Categorical resolutions of curves and Bridgeland stability}
\author{Nicolás Vilches}
\address{Department of Mathematics, Columbia University, 2990 Broadway, New York, NY 10027, USA}
\email{nivilches@math.columbia.edu}
\begin{abstract}
Categorical resolutions of singularities are a replacement of resolution of singularities within the realm of triangulated categories. They allow the study of the derived category of a singular variety $X$ via a triangulated category that behaves like the derived category of a smooth variety.

We follow these ideas to study the bounded derived category of a singular, reduced curve $C$ (with arbitrary singularities and number of components). We start by describing an explicit categorical resolution of singularities, specializing a general construction of Kuznetsov and Lunts. We prove the existence of Bridgeland stability conditions on these categories. As a consequence, we get the existence of proper, good moduli spaces of semistable objects. If the curve $C$ is irreducible, then we relate these moduli spaces to the moduli of slope-semistable torsion-free sheaves on $C$, and to the moduli of slope-semistable vector bundles on the (geometric) resolution $\tilde{C}$. This extends classical constructions by Oda and Seshadri, Bhosle and many others.

Finally, we use these results to give explicit descriptions of the moduli of torsion-free sheaves on a curve with a single node, cusp, or tacnode.
\end{abstract}
\maketitle

\tableofcontents

\section{Introduction}

Moduli spaces of vector bundles on curves are one of the most studied objects in algebraic geometry. Given a smooth, projective curve $C$, there is a moduli space $M_C(r, d)$ parametrizing (S-equivalence classes of) slope-semistable vector bundles on $C$ with rank $r$ and degree $d$. Their basic structure is well understood: for $(r, d)$ coprime and $g \geq 2$, they are smooth, projective moduli spaces of dimension $r^2(g-1)+1$. We refer to \cite{Ses82} for a discussion on their basic properties. 

The situation varies significantly if $C$ is allowed to have singularities, as the moduli space of semistable vector bundles becomes non-projective. Instead, we consider the moduli space of pure dimension one sheaves. And although the resulting moduli spaces $M_C(r, d)$ are projective, they are no longer smooth (even in the case $\gcd(r, d)=1$), and they can have multiple components (cf. \citelist{\cite{AIK77} \cite{Tei91}}).

Nevertheless, much progress has been made to unveil the structure of the moduli spaces $M_C(r, d)$. The basic strategy is to relate torsion-free sheaves on $C$ with vector bundles on the resolution $s\colon \tilde{C}\to C$, plus some additional gluing data. This has been extensively studied when $C$ has nodes; see for example \citelist{\cite{Ses82} \cite{OS79} \cite{Tei91} \cite{BB14} \cite{FB21}}.

Our goal is to discuss a different approach from the point of view of derived categories. The first technical issue that arises is that $\DC^b(\tilde{C})$ is inadequate to study $\DC^b(C)$, as the pullback $Ls^\ast\colon \DC^b(C) \to \DC^-(\tilde{C})$ is not fully faithful.

One way of circumventing this difficulty is to use a \emph{categorical resolution of singularities}: a smooth dg-category $\D$ endowed with functors $R\pi_\ast\colon \D \to \DC(C)$, $L\pi^\ast\colon \DC(C) \to \D$. satisfying $R\pi_\ast \circ L\pi^\ast = \id$. We refer to \citelist{\cite{Lun01} \cite{Lun10} \cite{KL15} \cite{KK15} \cite{BR11} \cite{BDG01}} for various discussions of this notion, and to Subsection \ref{subsec:catres}.

For singular curves, a small modification of the general procedure of \cite{KL15} gives a categorical resolution of singularities $R\pi_\ast\colon \DC^b(\Rcal) \to \DC^b(C)$, depending on the choice of a \emph{non-rational locus} $T \subset C$. This is a subscheme of $C$, supported at $\Sing(C)$, that measures the difference between $C$ and its resolution $\tilde{C}$. We point out that the scheme-theoretic support of $\rho_\ast \O_{\tilde{C}}/\O_C$ is contained in any non-rational locus. We also fix some integer $e$ such that $I^e=0$, where $I$ is the radical of $T$.

From here, we wish to find an adequate replacement of slope semistability on the categorical resolution. The most natural option is to look for a Bridgeland stability condition on $\DC^b(\Rcal)$. Not only this gives a notion of semistability for objects on $\DC^b(\Rcal)$; it also gives us the existence of moduli spaces. In our case, we get the following result.

\begin{teo} \label{teo:intro_main}
Let $C$ be a reduced, projective curve, and let $T, e$ be as above.
\begin{enumerate}
\item There exists a categorical resolution of singularities $\DC^b(\Rcal) \to \DC^b(C)$, where $\Rcal$ is a sheaf of non-commutative algebras on $C$ depending on $T$ and $e$.
\item We have that $K^{\num}(\Rcal) \cong \Z^{2n+me}$, where $n$ is the number of irreducible components of $C$, and $m$ the number of singular points. There is a distinguished basis of $K^{\num}(\Rcal)^\vee$ denoted by $E\mapsto \deg_k(E), \rk_k(E), \ell_{ij}(E)$ for $1 \leq k \leq n$, $1 \leq i \leq m$, $1 \leq j \leq e$. 
\item There is a heart of a bounded t-structure $\Cheart \subseteq \DC^b(\Rcal)$ and an open, connected subset $\Omega \subset \R^{3n+me}$ satisfying the following. For any $(\alpha_k, \beta_k, \gamma_k, \delta_{ij}) \in \Omega$, the pair $\sigma=(\Cheart, Z)$ defines a Bridgeland stability condition on $\DC^b(\Rcal)$, where $Z$ is given by the formula
\[ Z(E) = \sum_{i=1}^m \sum_{j=1}^e \delta_{ij} \ell_{ij}(E) - \sum_{k=1}^n \gamma_k \deg_k(E) + \sum_{k=1}^n \beta_k \rk_k(E) + i \sum_{k=1}^n \alpha_k \rk_k(E). \]
\item For any $\sigma$ as in (3) and any $v \in K^{\num}(\Rcal)$, there is an algebraic stack $\M_{\sigma}(v)$ parametrizing $\sigma_t$-semistable objects with numerical class $v$, admitting a proper, good moduli space $M_{\sigma}(v)$.
\end{enumerate}
\end{teo}

\begin{proof}
Part (1) is Subsection \ref{subsec:catres}; part (2) is Subsection \ref{subsec:Kgrp}; part (3) follows from Corollary \ref{cor:quad_stab}; part (4) is Theorem \ref{teo:moduli_main}. 
\end{proof}

We point out that $C$ is allowed to have arbitrary singularities. In particular, $C$ does not need to have locally planar singularities. 

One should think of the moduli spaces $M_\sigma(v)$ as ``interpolating'' between moduli spaces of torsion-free sheaves on $C$ and moduli spaces of vector bundles on $\tilde{C}$. One justification for this philosophy is given by the following theorem.

\begin{teo}
Let $C$ be a reduced, irreducible, projective curve. Fix $T, e$ as above.
\begin{enumerate}
\item There exists a connected, open subset $\tilde{\Omega} \subset \R^{me}$, and a continuous map $\Sigma\colon \tilde{\Omega} \to \Stab(\Rcal)$, mapping $\Delta=(\delta_{ij}) \in \tilde{\Omega}$ to the stability condition $\sigma=\sigma_\Delta$ with central charge
\[ Z_{\Delta}(E) = \sum_{i=1}^m \sum_{j=1}^e \delta_{ij} \ell_{ij}(E) - \deg(E) + i \rk(E). \]
\end{enumerate}
Now, fix $v$ be a numerical vector with $\rk(v)>0$. Assume that $v$ is primitive. 
\begin{enumerate}[resume]
\item There are finitely many walls for $v$ in $\tilde{\Omega}$, and each one of them defines a (non-zero) hyperplane in $\tilde{\Omega}$. In particular, for $\sigma \in \tilde{\Omega}$ outside of these hyperplanes, we have that an object $E$ with numerical class $v$ is $\sigma$-semistable if and only if it is $\sigma$-stable.
\item If $0 < \delta_{ij} \ll 1$ for all $i, j$, then there is a proper map $M_\sigma(v) \to M_{\tilde{C}}(r, d)$. 
\item If $0 < 1-\delta_{ie} \ll 1$ for all $i$, then there is a proper map $M_\sigma(v) \to M_C(r, \overline{d})$, where $\overline{d}=d+r(p_a(C)-g)-\sum_i \ell_{ie}$. 

\item Assume that $v=v(L\pi^\ast E)$, where $E$ is a vector bundle on $C$. Then, the map $M_\sigma(v) \to M_C(r, \overline{d})$ constructed above is an isomorphism over the locus of $M_C(r, \overline{d})$ parametrizing vector bundles. 
\end{enumerate}
\end{teo}

\begin{proof}
Part (1) is a consequence of Theorem \ref{teo:intro_main}; see also the discussion at the beginning of Subsection \ref{sec:comparison}. The rest is Theorem \ref{teo:comparison_main}.
\end{proof}

This result gives us a clear strategy to describe the moduli space $M_C(r, \overline{d})$. First, we set $v=v(L\pi^\ast E)$, where $E \in M_C(r, \overline{d})$. This way, we look at the moduli spaces $M_\sigma(v)$, as we vary $\sigma \in \tilde{\Omega}$. Parts (2) and (4) of Theorem \ref{teo:intro_comparison} relate these moduli spaces with $M_{\tilde{C}}(r, d)$ and $M_C(r, \overline{d})$, respectively. From here, we need to carefully describe the modifications at each wall.

We give two sample applications of these ideas. The first one involves a curve with a single node. This recovers classical descriptions, cf. \citelist{\cite{OS79} \cite{Ses82}*{Chapter 8} \cite{Bho96}}.

\begin{teo} \label{teo:intro_comparison}
Let $C$ be an irreducible curve of genus $g\geq 1$ with a single node or cusp. Let $\DC^b(\Rcal) \to \DC^b(C)$ be the categorical resolution of singularities from Theorem \ref{teo:intro_main}. Then there exists a path $(\sigma_t)_{t \in (0,1)}$ of stability conditions on $\DC^b(\Rcal)$ satisfying the following. 

Fix $v$ a primitive numerical vector with $r=\rk(v)>0$. Denote $d=\deg(v)$, $\ell=\ell(v)$. For each $t$, we let $M_t(v)=M_{\sigma_t}(v)$.
\begin{enumerate}
\item As we vary $t \in (0, 1)$, the moduli spaces $M_t(v) = M_{\sigma_t}(v)$ are birational. Moreover, for all but finitely many values of $t$, we have that $M_t(v)$ is smooth, projective of dimension $r^2( g-1) +1+\ell(2r-\ell)$.
\item For $t$ close to zero, we have a proper map $M_t(v) \to M_{\tilde{C}}(r, d)$. If $\gcd(r, d)=1$, this map is a $\Gr(2r, \ell)$-bundle, locally trivial in the Zariski topology. 
\item For $t$ close to 1, we have a proper map $R\pi_\ast\colon M_t(v) \to M_C(r, d+r-\ell)$. If $\ell=r$, this map is birational.
\item In the special case $r=\ell=1$, the moduli space $M_t(v)$ is independent of $t$; the map $M_t(v) \to M_{\tilde{C}}(1, d)$ is a $\P^1$-bundle; and the map $M_t(v) \to M_C(1, d)$ is birational.
\item In the special case $r=\ell=2$ and $d$ odd, there is a single wall at $t=1/2$. The two moduli spaces $M_{\epsilon}(v)$ and $M_{1-\epsilon}(v)$ are related by replacing a $\P^g$-bundle $P \to M_{1/2}(u) \times M_{1/4}(u)$ with a $\P^{g+1}$-bundle $P' \to M_{1/2}(u) \times M_{1/4}(w)$, where $M_{1/2}(u), M_{1/2}(w)$ are $g$-dimensional tori. The induced map $M_{1-\epsilon}(v) \dashrightarrow M_\epsilon(v)$ is a standard flip. 
\end{enumerate}
\end{teo}

\begin{proof}
See Theorems \ref{teo:node_main} and \ref{teo:noderktwo_main}.
\end{proof}

Our second application involves a curve with a \emph{tacnode}: a singularity of the form $\C[[x, y]]/(y^2-x^4)$. Contrary to the scenario for nodes, the non-rational locus in this situation is non-reduced, cf. Subsection \ref{subsec:nrlocal}.

\begin{teo}
Let $C$ be an irreducible curve of genus $g\geq 1$ with a single tacnode. Let $\DC^b(\Rcal) \to \DC^b(C)$ be the categorical resolution of singularities from Theorem \ref{teo:intro_main}. We let $\sigma_{\delta_1, \delta_2}$ be the stability conditions from Theorem \ref{teo:intro_comparison}, for $0 < \delta_1, \delta_2$, $\delta_1+\delta_2<1$.

Fix $v$ a primitive numerical vector with $r=\rk(v)>0$. Denote $d=\deg(v)$, $\ell=\ell(v)$. We denote by $M_{\delta_1,\delta_2}(v) = M_{\sigma_{\delta_1\delta_2}}(v)$.
\begin{enumerate}
\item We have $M_{\delta_1,\delta_2}(v) = \varnothing$ unless $0 \leq \ell_1 \leq 2r$, $0 \leq \ell_2 \leq 4r$, $0 \leq \ell_2-\ell_1 \leq 2r$. 
\end{enumerate}
From now on, we assume that $M_{\delta_1,\delta_2}(v) \neq \varnothing$ for some $\sigma$.
\begin{enumerate}[resume]
\item There are finitely many $v$-walls. If $\sigma_{\delta_1, \delta_2}$ and $\sigma_{\delta_1', \delta_2'}$ are outside of the walls, then $M_{\sigma_{\delta_1, \delta_2}}(v)$ and $M_{\sigma_{\delta_1', \delta_2'}}(v)$ are birational. 
\item For $\delta_1, \delta_2 \ll 1$, the map $M_{\delta_1, \delta_2}(v) \to M_{\tilde{C}}(r, d)$ is a bundle, locally trivial in the Zariski topology, with fiber
\[ \{ (V_1, V_2): V_1 \subset \C^{\oplus 2r}, V_2 \subset \C^{\oplus 4r}, V_1 \times \{0\} \subset V_2 \subset \C^{2r} \times V_1, \dim(V_i) = \ell_i \}. \]

\item Assume that $v=(1, d, 1, 2)$. Then, the moduli space $M_{\delta_1, \delta_2}(v)$ is independent of $(\delta_1, \delta_2)$. We have a diagram
\[ M_{\tilde{C}}(1, d) \xleftarrow{a_\ast} M_{\delta_1, \delta_2}(v) \xrightarrow{R\pi_\ast} M_C(1, d), \]
where the map $a_\ast$ is an $\mathbb{F}_2$-bundle, and the map $R\pi_\ast$ is birational. 

\item Assume that $v=(2, d, 2, 4)$ with $d$ odd. Then, there is a diagram
\[ M_{\tilde{C}}(2, d) \leftarrow M_1 \dasharrow M_2 \dasharrow M_3 \dasharrow M_4 \to M_C(2, d), \]
where $M_1 \to M_{\tilde{C}}(2, d)$ is a bundle, locally trivial in the Zariski topology; $M_i \dashrightarrow M_{i-1}$ is a standard flip over a $2g$-dimensional torus; and $M_4 \to M_C(2, d)$ is birational. 
\end{enumerate}
\end{teo}

\begin{proof}
See Lemma \ref{lemma:tacbasic_main} and Theorem \ref{teo:tac_main}.
\end{proof}

\subsection{Structure of the paper}

We will devote the first five sections to the construction of the stability conditions of Theorem \ref{teo:intro_main}. We start by reviewing the notion of a non-rational locus in Section \ref{sec:nr}, and relating it to the \emph{conductor} of a singular curve. 

The next two sections will construct a categorical resolution of $\DC^b(C)$, adapting the argument of \cite{KL15}. We first recall the construction of a generalized Auslander algebra in Section \ref{sec:aus}, which provides a categorical resolution of a zero-dimensional scheme. We use this to construct the categorical resolution of $\DC^b(C)$ in Section \ref{sec:catres}. 

We finish the first part by constructing the stability conditions in Section \ref{sec:stab}. The existence of moduli spaces is proven in Section \ref{sec:moduli}.

The last three sections are focused on applications. We prove general comparison results in Section \ref{sec:comparison}, including the proof of Theorem \ref{teo:intro_comparison}. We specialize to the case of a curve with a node or cusp in Section \ref{sec:node}, and of a curve with a tacnode in Section \ref{sec:tac}.

\subsection{Conventions}

We work over the complex numbers. If $V$ is a vector space of dimension $b$, and $1 \leq a \leq b$, we denote by $\Gr(V, a)$ the Grassmannian of $a$-dimensional \emph{quotients} of $V$. 

If $C$ is an irreducible projective curve, we write $M_C(r, d)$ for the moduli space of pure dimension one sheaves on $C$ with rank $r$ and degree $d$.

\subsection{Acknowledgements}

I am deeply thankful to my PhD advisor, Giulia Saccà, for her guidance and constant support during all these years. I am grateful of discussions with Arend Bayer, Amal Mattoo, Saket Shah, Sofia Wood, and Fan Zhou about various pieces of this project. I am especially thankful to Emanuele Macrì for sparkling my interest in categorical resolutions, and to Angela Ortega for pointing out the reference \cite{BRH21}.

This work was partially supported by the National Science Foundation (grant number DMS-2052934) and by the Simons Foundation (grant number SFI-MPS-MOV-00006719-09).


\section{Non-rational loci of curves and conductors} \label{sec:nr}

In this section we discuss the notion of a \emph{non-rational locus}, following \cite{KL15}*{\textsection 6.1}. We first recall its definition in Subsection \ref{subsec:nrgen}. In the case of curves, we relate this notion to the \emph{conductor} of a curve. After that, we compute explicit examples of non-rational loci for curves in Subsection \ref{subsec:nrlocal}.

\subsection{Generalities} \label{subsec:nrgen}

\begin{defin}[\cite{KL15}*{Definition 6.1}] \label{defin:nrgen_main}
Let $f\colon X \to Y$ be a proper, birational morphism. A subscheme $T \subseteq Y$ is a \emph{non-rational locus} of $Y$ with respect to $f$ if the canonical morphism $I_T \to Rf_\ast(I_{f^{-1}(T)})$ is an isomorphism, where $f^{-1}(T)$ is the scheme-theoretic preimage of $T$.
\end{defin}

Note that non-rational loci are not unique. For example, if $f$ is an isomorphism, then any subscheme is a non-rational locus. We also point out that being a non-rational locus is Zariski-local on $Y$.

\begin{remark}[\cite{KL15}*{Remark 6.2}]
If $f\colon X \to Y$ is a resolution of singularities, then $Y$ has rational singularities if and only if $T=\varnothing$ is a non-rational locus for $f$. In particular, if $C$ is a singular curve and $\rho\colon \tilde{C} \to C$ its resolution, then $T=\varnothing$ is \emph{never} a non-rational locus of singularities.
\end{remark}

\begin{defin}
Let $C$ be a reduced curve. A \emph{non-rational locus} of $C$ is a non-rational locus for the normalization map $\rho\colon \tilde{C} \to C$.
\end{defin}

\begin{lemma} \label{lemma:nrgen_existence}
Let $C$ be a reduced curve, and let $\rho\colon \tilde{C} \to C$ be its normalization. Then there is a non-rational locus $T \subset C$ whose set-theoretic support equals $\Sing(C)$. 
\end{lemma}

\begin{proof}
Note that $\rho$ is a blow-up of $C$ at a subscheme $T_0 \subset C$ set-theoretically supported at $\Sing(C)$, e.g. by \cite{Har77}*{Exercise II.7.11(c)}. The argument of \cite{KL15}*{Lemma 6.3} ensures that an infinitesimal thickening of $T_0$ is a non-rational locus for $\rho$. 
\end{proof}

\begin{remark}
Let $C$ be a reduced curve, $\rho\colon \tilde{C} \to C$ be its resolution of singularities, and let $T$ be a non-rational locus for $C$. If $S=\rho^{-1}(T)$ is its scheme-theoretic preimage, then we have a short exact sequence $0 \to \rho_\ast I_S \to \rho_\ast \O_{\tilde{C}} \to \rho_\ast \O_S \to 0$. Using that $T$ is a non-rational locus gives us the diagram
\begin{equation} \label{eq:nrgen_keydiagram}
\begin{tikzcd}
& 0 \arrow[d] & 0 \arrow[d] \\
& I_T \arrow[r, "\cong"] \arrow[d] & \rho_\ast I_S \arrow[d] \\
0 \arrow[r] & \O_C \arrow[r] \arrow[d] & \rho_\ast \O_{\tilde{C}} \arrow[r] \arrow[d] & Q \arrow[r] \arrow[d, equal] & 0 \\
0 \arrow[r] & \O_T \arrow[r] \arrow[d] & \rho_\ast \O_S \arrow[r] \arrow[d] & Q \arrow[r] & 0 \\
& 0 & 0
\end{tikzcd}
\end{equation}
In particular, we point out that $\O_C$ is the kernel of a map $\rho_\ast \O_{\tilde{C}} \to Q$; the isomorphism $Q \cong \rho_\ast \O_S/\O_T$ ensures that the scheme-theoretic support of $Q$ is contained in $T$. 
\end{remark}

Let us finish up this section by relating this notion with the conductor of a curve. Recall that \emph{conductor} of a reduced curve $C$ is the ideal $J=\Ann_C(\rho_\ast \O_{\tilde{C}}/\O_C)$.

\begin{lemma} \label{lemma:nrgen_conductorisnr}
Let $C$ be a reduced curve. If $T$ is a non-rational locus for $\rho\colon \tilde{C} \to C$, and $I=I(T)$ is the ideal corresponding to $T$, then $I \subseteq J$. Moreover, we have that $J$ is a non-rational locus.
\end{lemma}

In other words, this lemma ensures that the closed subscheme determined by $J$ is the smallest non-rational locus for $C$.

\begin{proof}
For the first part, note from \eqref{eq:nrgen_keydiagram} that we have the isomorphism $\rho_\ast \O_{\tilde{C}}/\O_C = \rho_\ast \O_S/\O_T$. The latter is scheme-theoretically supported at $T$, and so $I(T) \subseteq \Ann_C(\rho_\ast \O_{\tilde{C}}/\O_C)$. The second part follows from the well-known fact that $J$ is also an ideal in $\O_{\tilde{C}}$.
\end{proof}

\subsection{Local description} \label{subsec:nrlocal}

Let $C$ be a reduced curve, and let $\rho\colon \tilde{C} \to C$ be its resolution. We have shown that there exists a non-rational locus of $C$ with respect to $\rho$, supported set-theoretically on $\Sing(C)$. It is convenient to have an explicit description of a non-rational locus in terms of the singularities of $C$.

Up to shrinking $C$, we may assume that $C=\Spec A$ is affine, and that $p$ is the only singular point of $C$. We use the following key observation.

\begin{lemma}
Let $C$ be an affine curve with $q \in C$ its only singular point. Denote by $\rho\colon \tilde{C} \to C$ the resolution of $C$, $C_q$ the completion of $C$ at $q$, and $\rho_p\colon \tilde{C}_q \to C_q$ the base change.

Given a closed subscheme $T \subseteq C$ supported set-theoretically at $p$, we have that $T$ is a non-rational locus for $C$ if and only if the corresponding subscheme $T_q \subseteq C_q$ is a non-rational locus for $\rho_q$.
\end{lemma}

\begin{proof}
Denote by $\phi\colon C_q \to C$ the completion. We have that $\phi^\ast I_T = I_{T_q}$, and similarly for the preimages. The result follows immediately from the fact that the completion functor is faithfully flat. 
\end{proof}

As a direct corollary, non-rational loci depend only on the germ of the singularity. This helps greatly in calculations.

\begin{example}[Ordinary $n$-uple points]
Recall that an \emph{ordinary $n$-uple point} is the germ of the curve $C=\Spec \C[x, y]/((y-\lambda_1 x)\dots (y-\lambda_n x))$ at the origin, where $\lambda_1, \dots, \lambda_n$ are pairwise distinct. 

Note that $\tilde{C} = \Spec \C[x, t]/(t-\lambda_1)\dots (t-\lambda_n)$. This way, we quickly check that the conductor of $C$ is $(x, y)^{n-1}$.
\end{example}

\begin{example}[$A_k$ singularities]
Recall that an $A_k$ singularity is the germ of $C=\Spec \C[x, y]/(y^2-x^{k+1})$ at the origin. To compute its conductor, we separate in cases:
\begin{itemize}
\item If $k$ is even, then the curve has a single branch at the origin. The normalization is given by $\mathbb{A}^1 \to C$, $t \mapsto (t^2, t^{k+1})$. From here, one quickly checks that $(x^{k/2}, y)$ is the conductor. 
\item If $k$ is odd, then the curve has two branches at the origin. The normalization is $\mathbb{A}^1_s \sqcup \mathbb{A}^1_t \to C$, $s \mapsto (s, s^{(k+1)/2})$ and $t \mapsto (t, -t^{(k+1)/2})$. We directly verify that $(x^{(k+1)/2}, y)$ is the conductor.
\end{itemize}
Putting everything together, we get that $I=(x^{\lfloor (k+1)/2 \rfloor}, y)$ is the conductor.
\end{example}

We have summarized our computations in Table \ref{table:nrlocal_planar}.

\begin{table}[htbp]
\centering
\caption{Conductors of some planar singularities.}
\label{table:nrlocal_planar}
\begin{tabular}{|c|c|c|} \hline
Singularity & Local equation & Conductor \\ \hline \hline
Node & $xy$ & $(x, y)$ \\
Cusp & $y^2-x^3$ & $(x, y)$ \\
Ordinary $n$-uple point & $\prod_{i=1}^n (y-\lambda_i x)$ & $(x, y)^n$ \\
$A_k$ & $x^{k+1} + y^2$ & $(x^{\lfloor (k+1)/2\rfloor}, y)$ \\ \hline
\end{tabular}
\end{table}

\begin{example} (Union of coordinate axes)
Let $C$ be the curve obtained as the union of the $n$ coordinate axis of $\C^n$. We will check in Proposition \ref{prop:nrsmall_classification} that $T=\{(0,\dots,0)\}$ is a non-rational locus for $C$. 
\end{example}

\subsection{Small non-rational locus} \label{subsec:nrsmall}

Let us finish by classifying which singular curves $C$ admit $\Sing(C)$ with its reduced structure as a non-rational locus. Note by Lemma \ref{lemma:nrgen_conductorisnr}, this is the same as asking that $\rho_\ast \O_{\tilde{C}}/\O_C$ is scheme-theoretically supported at $\Sing(C)$.

\begin{prop} \label{prop:nrsmall_classification}
Let $C$ be a reduced curve with a single singular point $q \in C$, and let $\rho\colon \tilde{C} \to C$ be its resolution of singularities. Assume that $\{q\}$ (with its reduced structure) is a non-rational locus of $C$. If $S=\rho^{-1}(q)$ is a subscheme supported at $s$ points with lengths $e_1, \dots, e_s$, then
\begin{equation} \label{eq:nrsmall_ring}
\hat{\O}_{C, p} \cong \frac{\C[[x_{i, j}: 1 \leq i \leq s, e_i \leq j \leq 2e_i-1]]}{(\{x_{i,j}^{j'} - x_{i, j'}^j\}_{e_i \leq j, j' \leq 2e_i-1}, \{x_{i,j}x_{i', j'}\}_{i \neq i'})}.
\end{equation}
\end{prop}

\begin{proof}
Denote by $p_1, \dots, p_s$ the preimages of $q \in C$. This way, we can write the scheme-theoretic preimage $\rho^{-1}(q)$ as $\bigcup_{i=1}^s \{e_i p_i\}$ for some $e_i \geq 1$. 

Let us write $A=\hat{\O}_{C, q}$, and let $\tilde{A} = \hat{\O}_{\tilde{C}, \rho^{-1}(q)}$. Here $\tilde{A} = \prod_{i=1}^s \C[[t_i]]$, as $\tilde{C}$ is smooth and $\rho^{-1}(q)$ consists of $s$ points. This way, the diagram \eqref{eq:nrgen_keydiagram} gives us the Cartesian square:
\[ \begin{tikzcd} A \arrow[r] \arrow[d] & A/I_q \cong \C \arrow[d, hook] \\ \prod_{i=1}^k \C[[t_i]] \arrow[r] & \prod_{i=1}^s \C[t_i]/(t_i^{e_i}) 
\end{tikzcd} \]
We get that $A \to \prod_{i=1}^k \C[[t_i]]$ embeds $A$ into the subring 
\[ \left\{ \sum_{i=0}^n \sum_{j=0}^\infty a_{ij} t_i^j : a_{0, 1}=\dots=a_{0, n}, a_{i,j}=0 \text{ if } 0<j<e_j\right\}.\]
As the map $A \to \prod_{i=1}^k \C[[t_i]]$ is a ring homomorphism, this also determines the ring structure on $A$. The claimed result follows immediately by picking generators $x_{i,j} = t_i^j$.
\end{proof}

Note that the isomorphism class of a singularity $q \in C$ as in Proposition \ref{prop:nrsmall_classification} depends only on the tuple $(e_1, \dots, e_s)$ of multiplicities, up to rearrangement. For instance, the tuples $(1, 1)$ and $(2)$ correspond to a node and a cusp, respectively. Similarly, the tuple $(1, 1, \dots, 1)$ corresponds to the singularity formed by taking the coordinate axes in $\C^s$. 

If $q \in C$ is a singularity with multiplicity tuple $(e_1, \dots, e_s)$, the description \eqref{eq:nrsmall_ring} ensures that the embedding dimension of $C$ at $q$ is $\sum_{i=1}^s e_i$. In particular, the only planar singularities that arise in this fashion are the node and the cusp.

\section{Generalized Auslander algebras for zero-dimensional schemes} \label{sec:aus}

Following \cite{KL15}*{\textsection 5}, we recall the construction of a (generalized) Auslander algebra in Subsection \ref{subsec:ausalg}. We later describe various functors in Subsection \ref{subsec:ausfunctors}.

\subsection{Auslander algebras and their modules} \label{subsec:ausalg}

Our setup is the following. We let $Z$ be a zero-dimensional affine scheme. We fix an ideal $I$ on $Z$, and an integer $n \geq 1$ such that $I^n=0$. (Note that we are \emph{not} requiring that $I^{n-1}\neq 0$.)

\begin{defin}[\cite{KL15}*{pp. 4579--80}]
Set $\A=\A_{S, I, n}$ to be the sheaf of algebras over $Z$ described as
\[ \A = \begin{pmatrix} \O_Z & I & I^2 & \cdots & I^{n-1} \\ \O_Z/I^{n-1} & \O_Z/I^{n-1} & I/I^{n-1} & \cdots & I^{n-2}/I^{n-1} \\ \O_Z/I^{n-2} & \O_Z/I^{n-2} & \O_Z/I^{n-2} & \cdots & I^{n-3}/I^{n-2} \\ \vdots & \vdots & \vdots & \ddots & \vdots \\ \O_Z/I & \O_Z/I & \O_Z/I & \cdots & \O_Z/I \end{pmatrix}, \]
where the algebra structure is given as a subset of $\End_Z(\O_Z \oplus \O_Z/I^{n-1} \oplus \dots \oplus \O_Z/I)$.

We let $\Coh(\A)$ be the category of (right) $\A$-modules that are coherent over $Z$. 
\end{defin}

\begin{lemma}[\cite{KL15}*{Lemma A.9}] \label{lemma:ausalg_quiver}
The category of coherent $\A$-modules is equivalent to the following category. 

Objects are tuples $(M_1, \dots, M_n, \alpha, \beta_1, \dots, \beta_{n-1})$, where the $M_i$ are sheaves on $Z$, $\alpha\colon M_i \to M_{i+1}$, and $\beta_k\colon M_i \otimes I^k \to M_{i-k}$. We impose the following compatibility conditions.
\begin{enumerate}
\item The diagram 
\[ \begin{tikzcd} & M_1 \otimes I \arrow[r, "\alpha \otimes \id"] \arrow[d, "\times"'] \arrow[ld] & M_2 \otimes I \arrow[r, "\alpha \otimes \id"] \arrow[d, "\times"'] \arrow[ld, "\beta_1"'] & \cdots \arrow[r, "\alpha \otimes \id"] \arrow[ld, "\beta_1"'] & M_n \otimes I \arrow[d, "\times"'] \arrow[ld, "\beta_1"'] \\ 0 \arrow[r] & M_1 \arrow[r, "\alpha"] & M_2 \arrow[r, "\alpha"] & \cdots \arrow[r, "\alpha"] & M_n \end{tikzcd} \]
commutes, where $\times$ is the multiplication map.
\item We have that $\beta_k \beta_\ell = \beta_{k+\ell}$. More precisely, we have that the diagram 
\[ \begin{tikzcd} M_i \otimes I^k \otimes I^\ell \arrow[r, "\id \otimes \times"] \arrow[d, "\beta_k\otimes \id"'] & M_i \otimes I^{k +\ell} \arrow[d, "\beta_{k+\ell}"] \\ M_{i-k} \otimes I^\ell \arrow[r, "\beta_\ell"'] & M_{i-k-\ell} \end{tikzcd} \]
commutes.
\end{enumerate}
Morphisms are given by maps $\phi_i\colon M_i \to M_i'$ compatible with the $\alpha$ and $\beta_k$.
\end{lemma}

We point out that condition (2) ensures that the maps $\beta_k, k\geq 2$ are determined by $\beta_1$, as $I^k \otimes I^\ell \to I^{k+\ell}$ is surjective. Compare to \cite{KL15}*{Definition A.8(2)}.

We will drop the maps $\alpha$ and $\beta$, and simply write $M=(M_1, \dots, M_n)$ for an object of $\Coh(\A)$.

\subsection{Functors} \label{subsec:ausfunctors}

In this subsection we will construct various functors on the category of coherent $\A$-modules. To start, we define the \emph{resolution} functors as follows.

\begin{defin}[\cite{KL15}*{p. 4615}] \label{defin:ausfunctors_res}
We define functors $\mu_Z^\ast\colon \Coh(Z) \to \Coh(\A)$ and $\mu_{Z\ast}\colon \Coh(\A) \to \Coh(Z)$ via the formulas
\[ \mu_{Z\ast}(M_1, \dots, M_n) = M_n, \quad \mu_Z^\ast(N) = (N \otimes_Z I^{n-1}, N \otimes_Z I^{n-2}, \dots, N), \]
where the maps $\alpha$ are induced by inclusions, and the maps $\beta$ by the multiplication of elements of $I$. 
\end{defin}

Note from the definition that $\mu_{Z\ast}$ is exact, while $\mu_Z^\ast$ is right exact. Moreover, the adjunction $\mu_Z^\ast \dashv \mu_{Z\ast}$ holds.

\begin{defin}[\cite{KL15}*{p. 4614}] \label{defin:ausfunctors_sodtruncation}
Assume that $n \geq 2$, and denote by $Z'$ the closed subscheme corresponding to $I^{n-1}$. Set $e_\ast\colon \Coh(\A_{Z', I, n-1}) \to \Coh(\A_{Z, I, n})$ and $e^!\colon \Coh(\A_{Z, I, n}) \to \Coh(\A_{Z', I, n-1})$ via the formulas
\[ e_\ast(M_1, \dots, M_{n-1}) = (M_1, \dots, M_{n-1}, M_{n-1}), \]
where $\alpha\colon M_{n-1} \to M_{n-1}$ is the identity and $\beta_1\colon M_{n-1} \otimes_Z I \to M_{n-1}$ is the multiplication; and 
\[ e^!(M_1, \dots, M_n) = (M_1, \dots, M_{n-1}). \]
\end{defin}

Here both $e^!$ and $e_\ast$ are exact. Moreover, we have the adjunction $e_\ast \dashv e^!$.

\begin{defin}[\cite{KL15}*{p. 4614}] \label{defin:ausfunctors_sodpt}
Let $Z_0 \subseteq Z$ be the closed subscheme corresponding to $I$. Define $i^\ast\colon \Coh(\A_{Z, I, n}) \to \Coh(Z_0)$ and $i_\ast\colon \Coh(Z_0) \to \Coh(\A_{Z, I, n})$ via the formulas
\[ i^\ast(M_1, \dots, M_n) = \coker(\alpha\colon M_{n-1} \to M_n), \qquad i_\ast(N) = (0, \dots, 0, N). \]
\end{defin}

Note that $\coker(\alpha\colon M_{n-1} \to M_n)$ is in fact a module over $Z_0$ (instead of merely $Z$), as the multiplication map $M_n \otimes I \to M_n$ factors through $\alpha$. We point out that we have the adjunction $i^\ast \dashv i_\ast$, that $i_\ast$ is exact, and that $i^\ast$ is right exact.

\begin{defin}[\cite{KL15}*{\textsection 5.2}] \label{defin:ausfunctors_sch}
Let $f\colon Z \to W$ be a morphism of zero-dimensional schemes. Let $I \subset \O_Z$, $J \subset \O_W$ be nilpotent ideals such that $f^{-1}(J) \subseteq I$. Pick $k$ such that $I^k = J^k =0$. 

We define functors $f_\ast\colon \Coh(\A_{Z, I, k}) \to \Coh(\A_{W, J, k})$ and $f^\ast\colon \Coh(\A_{W, J, k}) \to \Coh(\A_{Z, I, k})$ via the formulas
\[ f_\ast(M_1, \dots, M_n) = (f_\ast M_1, \dots, f_\ast M_k), \quad f^\ast (N_1, \dots, N_k) = (f^\ast N_1, \dots, f^\ast N_k). \]
\end{defin}

\subsection{Derived category}

Given $S, I, n$ as before, we consider the bounded derived category $\DC^b(\A)$ of $\Coh \A$. In this subsection we will review some of its properties, following \cite{KL15}*{\textsection 5.3--4}.

Let us start by discussing how to extend the functors of Subsection \ref{subsec:ausfunctors} to the derived category. In Definition \ref{defin:ausfunctors_res}, the functor $\mu_Z^\ast$ gives a left derived functor $L\mu_S^\ast\colon \DC^-(\A) \to \DC^-(Z)$, while $\mu_{Z\ast}$ descends directly to the (bounded or unbounded) derived category. In Definition \ref{defin:ausfunctors_sodpt}, the functor $i^\ast$ can be left derived at the level of \emph{bounded} derived categories, cf. \cite{KL15}*{pp. 4588--9}. In Definition \ref{defin:ausfunctors_sch}, the functor $f_\ast$ is exact (as $f$ is affine), while $f^\ast$ requires to be left derived. 

These functors can be used to effectively describe the derived category $\DC^b(\A)$, as the following result shows.

\begin{prop}[\cite{KL15}*{Proposition 5.14}] \label{prop:auscd_presod}
Let $n \geq 2$. There is a semi-orthogonal decomposition $\DC^b(\A_{Z, I, n}) = \langle i_\ast(\DC^b(Z_0)), e_\ast(\DC^b(\A_{Z', I, n-1})) \rangle$.
\end{prop}

\begin{cor}[\cite{KL15}*{Corollary 5.15}] \label{cor:ausdc_sod}
The derived category $\DC^b(\A_{Z, I, n})$ admits a semi-orthogonal decomposition with $n$ components equivalent to $\DC^b(Z_0)$.
\end{cor}

\section{Categorical resolution} \label{sec:catres}

In this section we discuss categorical resolutions of singularities. We refer to \citelist{\cite{Lun01} \cite{KL15} \cite{KK15}} for constructions of categorical resolutions of singularities in greater level of generality. Our goal is to describe a categorical resolution of singularities for a reduced curve $C$, following the approach of \cite{KL15}.

\subsection{Setup} \label{subsec:setup}

Let us fix some notation for the remainder of the section. We let $C$ be a reduced curve, with irreducible components $C_1, \dots, C_n$, and let $\Sing(C) = \{q_1, \dots, q_m\}$ be the set of singular points. We denote by $\rho\colon \tilde{C} \to C$ the normalization of $C$. The corresponding irreducible components are $\tilde{C}_1, \dots, \tilde{C}_n$. 

Fix a non-rational locus $T \subset C$ for the morphism $\rho$, set-theoretically supported at $\Sing(C)$. (Its existence is guaranteed by Lemma \ref{lemma:nrgen_existence}.) Denote by $S = \rho^{-1}(T)$ its scheme-theoretic preimage. Let $I$ be the ideal of $\{q_1, \dots, q_m\} \subset T$, and $J$ the ideal of $\rho^{-1}(\{q_1, \dots, q_m\}) \subset S$. Also, fix some exponent $e$ such that $I^e=0$.

\begin{remark}
Note from Lemma \ref{lemma:nrgen_conductorisnr} that we can take $T$ to be subscheme associated to the conductor of $C$. 
\end{remark}

For each $1\leq i \leq m$, we let $T_i$ be the irreducible component of $T$ supported at $q_i$, and $S_i = \rho^{-1}(T_i)$ its scheme-theoretic preimage. Denote by $J_i$ the ideal of $q_i \in T_i$, and $I_i$ the ideal of $S_i^\circ:=\rho^{-1}(q_i) \subset S_i$. 

We let $j\colon T \to C$ and $i\colon S \hookrightarrow \tilde{C}$ be the inclusions. We abuse notation by using $\rho\colon S \to T$ the induced morphism. These maps fit in the following diagram of Cartesian squares:
\begin{equation} \label{eq:setup_arrows}
\begin{tikzcd}
S_i^\circ \arrow[r, hook] \arrow[d] & S_i \arrow[r, hook] \arrow[d] & S \arrow[r, hook, "i"] \arrow[d, "\rho"] & \tilde{C} \arrow[d, "\rho"] \\ \{q_i\} \arrow[r, hook] & T_i \arrow[r, hook] & T \arrow[r, hook, "j"] & C.
\end{tikzcd}
\end{equation}

Finally, we write $\A=\A_{T, I, e}$ for the sheaf of algebras from Subsection \ref{subsec:ausalg}.

\subsection{An abelian category} \label{subsec:abcat}

\begin{defin}
We let $\Coh(\Rcal)$ be the category whose objects are triples $\F=(F, \overline{F}; \phi)$, where $F \in \Coh(\tilde{C})$, $\overline{F} \in \Coh(\A)$, and $\phi\colon \rho_\ast(\mu_S^\ast(F|_S)) \to \overline{F}$. 
\end{defin}

The notation here is chosen as one can see this category as the category of coherent modules over a sheaf of non-commutative algebras on $C$, cf. \cite{KL15}*{\textsection 5.1}. In particular, we point out that $\Coh(\Rcal)$ is an abelian category. 

Note that the morphism $\phi\colon \rho_\ast(\mu_S^\ast(F|_S)) \to \overline{F}$ can be equivalently described as morphisms $\phi_i\colon \rho_\ast(F|_S \otimes_S I^{e-i}) \to M_i$, such that the diagrams
\[ \begin{tikzcd} \rho_\ast(F|_S \otimes_S I^{e-i}) \arrow[r] \arrow[d, "\phi_i"'] & \rho_\ast(F|_S \otimes_S I^{e-(i+1)}) \arrow[d, "\phi_{i+1}"] \\ M_i \arrow[r, "\alpha"'] & M_{i+1} \end{tikzcd} \]
and
\[ \begin{tikzcd}
\rho_\ast(F|_S \otimes_S I^{e-(i+1)} \otimes_S I) \arrow[r, "\id\otimes \times"] \arrow[d, "\phi_{i+1} \otimes \id"'] & \rho_\ast(F|_S \otimes_S I^{e-i}) \arrow[d, "\phi_i"] \\ M_{i+1} \otimes_S I \arrow[r, "\beta_1"'] & M_i 
\end{tikzcd} \]
commute.

\begin{remark}
Suppose that $T=\{q_1, \dots, q_m\}$ with its reduced structure is a non-rational locus for $C$. (These cases were classified in Subsection \ref{subsec:nrsmall}.) We have that $I=0$, and so we can take $e=1$. Under this assumption, the category $\Coh(\Rcal)$ reduces to $\{(F, \overline{F}; \phi): F \in \Coh(C), \overline{F} \in \Coh(T), \phi\colon \rho_\ast(F|_S) \to \overline{F}\}$.
\end{remark}

\subsection{Functors} \label{subsec:functors}

In this subsection we will construct some functors relating $\Coh(\Rcal)$ to $\Coh(\tilde{C})$ and $\Coh(\A)$, and their extensions to the respective derived categories.

We start by constructing some functors between $\Coh(\Rcal)$ and $\Coh(\tilde{C})$. To start, set $a^\ast\colon \Coh(\tilde{C}) \to \Coh(\Rcal)$ and $a_\ast\colon \Coh(\Rcal) \to \Coh(\tilde{C})$ given by the formulas $a^\ast(M) = (M, \rho_\ast(\mu_S^\ast)(M|_S); \id)$ and $a_\ast(F, \overline{F}; \phi)=F$.

\begin{lemma}
The functor $a_\ast$ is exact, and the functor $a^\ast$ is right exact. Additionally, we have the adjunction $a^\ast \dashv a_\ast$.
\end{lemma}

\begin{proof}
The exactness of $a_\ast$ is clear. For $a^\ast$, it follows directly from composition of right-exact functors.

Now, consider $M \in \Coh(\tilde{C})$, $(F, \overline{F}; \phi) \in \Coh(\Rcal)$. Here $\Hom(a^\ast(M), (F, \overline{F}; \phi))$ consists of morphisms $M \to F$ and $\rho_\ast(\mu_S^\ast)(M|_S) \to \overline{F}$ making the diagram
\[ \begin{tikzcd} \rho_\ast(\mu_S^\ast)(M|_S) \arrow[r] \arrow[d, equal] & \rho_\ast(\mu_S^\ast)(F|_S) \arrow[d, "\phi"] \\ \rho_\ast(\mu_S^\ast)(M|_S) \arrow[r] & \overline{F} \end{tikzcd} \]
commute. This way,we have that
\[ \Hom(a^\ast(M), (F, \overline{F}; \phi)) \cong \Hom_{\tilde{C}}(M, F) = \Hom(M, a_\ast(F, \overline{F}; \phi)). \]
This shows the required adjunction.
\end{proof}

As a direct consequence, the functor $a_\ast$ descends directly to the (bounded) derived categories. However, to construct $La^\ast$ some extra care must be taken. Let us recall the following well-known result.

\begin{prop}[\cite{Har66}*{Corollary 5.3.$\beta$--$\gamma$}] \label{prop:functors_existence}
Let $F\colon \Bcat \to \mathcal{C}$ be a right exact additive functor between abelian categories. Assume that there is a class of objects $\overline{\Bcat}$ of $\Bcat$ satisfying the following properties:
\begin{enumerate}
\item Given any $B \in \Bcat$, there is a surjection $\overline{B} \to B \to 0$ with $\overline{B} \in \overline{\Bcat}$.
\item Given any short exact sequence $0 \to B \to B' \to B''$ in $\Bcat$ with $B'' \in \overline{\Bcat}$, we have that $B \in \overline{\Bcat}$ if and only if $B' \in \overline{\Bcat}$.
\item The functor $F$ preserves short exact sequences in $\overline{\Bcat}$.
\end{enumerate}
Then the left derived functor $LF\colon \DC^-(\Bcat) \to \DC^-(\mathcal{C})$ exists, and $LF(E)=F(E)$ for any bounded above complex with elements in $\overline{\Bcat}$.

Moreover, assume that condition (1) is replaced by
\begin{enumerate}
\item[(1')] There is a positive integer $r$ such that any $B \in \Bcat$ admits a resolution $0 \to B^{-r} \to B^{-r+1} \to \dots \to B^0 \to B \to 0$, with $B^i \in \Bcat$.  
\end{enumerate}
Then the functor $LF$ is defined at the level of unbounded categories, and it satisfies $LF(\DC(\Bcat)^{[i,j]}) \subseteq \DC(\mathcal{C})^{[i-r, j]}$ for all $i \leq j$. In particular, $LF$ maps $\DC^b(\Bcat)$ to $\DC^b(\mathcal{C})$.
\end{prop}

Of course, a dual version of Proposition \ref{prop:functors_existence} also holds, guaranteeing the existence of right derived functors.

\begin{lemma} \label{lemma:functors_Labounded}
The functor $a^\ast$ preserves short exact sequences of locally free sheaves in $\tilde{C}$. As a consequence, the left derived functor $La^\ast\colon \DC^b(\tilde{C}) \to \DC^b(\A)$ exists. Finally, we have the adjunction $La^\ast \dashv a_\ast$.
\end{lemma}

\begin{proof}
If $0 \to E \to E' \to E'' \to 0$ is a short exact sequence of locally free sheaves in $\tilde{C}$, then the restriction $0 \to E|_S \to E'|_S \to E''|_S \to 0$ is exact. Exactness is preserved after tensoring with $I^k$ (as the $E|_S$ are locally free on $S$), and so after applying $\mu_S^\ast(-)$. As $\rho_\ast$ is exact, this immediately gives us the first part. 

Now, note that the class of locally free sheaves in $\tilde{C}$ satisfy conditions (1') (with $r=1$) and (2) in Proposition \ref{prop:functors_existence}. This immediately gives us the second part. The third part follows directly from the adjunction $a^\ast \dashv a_\ast$.
\end{proof}

Let us now construct functors relating $\Coh(\Rcal)$ with $\Coh(\A)$. We start by considering the functors $b^\ast\colon \Coh(\A) \to \Coh(\Rcal)$ and $b_\ast \colon \Coh(\A) \to \Coh(Y)$ given by the formulas $b^\ast(N) = (0, N; 0)$ and $b_\ast(F, \overline{F}; \phi) = \overline{F}$. It is easy to see that $b^\ast$, $b_\ast$ are exact and $b^\ast \dashv b_\ast$. This way, they immediately induce exact functors at the level of derived categories, while also preserving the adjunction.

\begin{lemma}
The functor $b^\ast\colon \Coh(\A) \to \Coh(\Rcal)$ admits a left adjoint $b_!$, given by $b_!(F, \overline{F}; \phi) = \coker \phi$.
\end{lemma}

\begin{proof}
Let $\F=(F, \overline{F}; \phi) \in \Coh(\Rcal)$ and $N \in \Coh(\A)$ be given, and let us compute $\Hom(\F, b^\ast N)$. By definition, these are given by maps $f\colon F \to 0$ and $g\colon \overline{F} \to N$ making the diagram
\[ \begin{tikzcd} r_\ast F|_Z \arrow[r, "r_\ast f|_Z"] \arrow[d, "\phi"] & 0 \arrow[d] \\ \overline{F} \arrow[r, "g"] & N \end{tikzcd} \]
commute. In other words, these corresponds to maps $g\colon \overline{F} \to N$ such that $g \circ \phi=0$. These are uniquely determined by a morphism $\coker \phi \to N$, proving the required statement.
\end{proof}

\begin{lemma} \label{lemma:functors_Lbbounded}
The left derived functor $Lb_!\colon \DC^b(\A) \to \DC^b(Y)$ exists, and we have the adjunction $Lb_! \dashv b^\ast$.
\end{lemma}

\begin{proof}
Consider the class 
\begin{equation} \label{eq:functors_classP}
\Pscr = \{ (F, \overline{F}; \phi)\colon F \text{ locally free, } \ker \phi=0 \}.
\end{equation}
(If $\tilde{C}$ has multiple components, then $F$ is allowed to have different rank on each component; even $F=0$ is allowed.) We will prove that $\Pscr$ satisfies properties (1') (with $r=2$), (2) and (3) of Proposition \ref{prop:functors_existence}. From here, the existence of the left derived functor is immediate. Finally, $Lb_!\dashv b^\ast$ follows from $b_! \dashv b^\ast$. 
\end{proof}

\begin{proof}[Proof of Lemma \ref{lemma:functors_Lbbounded}, (1')]
Let $\F=(F, \overline{F}; \phi)$ be an element in $\Coh(\Rcal)$. Pick a resolution $0 \to E^{-1} \to E^0 \to F \to 0$, with $E^{-1}, E^0$ locally free. This way, if $\E^0=(E_0, \rho_\ast(\mu_S^\ast({E_0}|_S)) \oplus \overline{F}; \id \oplus \psi)$, we get a surjection $\E^0 \twoheadrightarrow \F$, whose kernel is of the form $(E^{-1}, K; \eta)$. Here $\E^0$ lies on $\Pscr$. 

Now, note that we have a surjection
\[ \E^{-1} = (E^{-1}, \rho_\ast(\mu_S^\ast(E^{-1}|_S)) \oplus K; \id \oplus \eta) \to (E^{-1}, K; \eta) \to 0, \]
where the map $\rho_\ast(\mu_S^\ast(E^{-1}|_S)) \oplus K \to K$ is $(0, \id)$. The kernel of this surjection is $\E^{-2} = (0, \rho_\ast(\mu_S^\ast(E^{-1}|_S)); 0)$, which clearly lies in $\Pscr$. This gives us the claimed resolution.
\end{proof}

\begin{proof}[Proof of Lemma \ref{lemma:functors_Lbbounded}, (2)] Consider a short exact sequence $0 \to \F \to \F' \to \F'' \to 0$. Write $\F=(F, \overline{F}; \phi)$ and similar for $\F', \F''$. We get a short exact sequence $0 \to F \to F' \to F'' \to 0$ in $\tilde{C}$, with $F''$ locally free. Here $F$ is locally free if and only if $F'$ is. Under this assumption, the restriction to $S$ and tensoring with $I^k$ is exact, and we get the commutative diagram
\begin{equation} \label{eq:functors_sesP}
\begin{tikzcd} 0 \arrow[r] & \rho_\ast(\mu_S^\ast(F|_S)) \arrow[r] \arrow[d, "\phi"] & \rho_\ast(\mu_S^\ast(F'|_S)) \arrow[r] \arrow[d, "\phi'"] & \rho_\ast(\mu_S^\ast(F''|_S)) \arrow[r] \arrow[d, "\phi''"] & 0 \\ 0 \arrow[r] & \overline{F} \arrow[r] & \overline{F}' \arrow[r] & \overline{F}'' \arrow[r] & 0. \end{tikzcd}
\end{equation}
By assumption the kernel of $\phi''$ is trivial. Thus, we get that $\ker \phi \cong \ker \phi'$ by the snake lemma, which allows us to conclude immediately. 
\end{proof}

\begin{proof}[Proof of Lemma \ref{lemma:functors_Lbbounded}, (3)]
We keep the notation of the previous proof. By assumption we have the diagram \eqref{eq:functors_sesP}, with $\ker \phi''=0$. The snake lemma gives us the short exact sequence $0 \to b_! \F \to b_! \F' \to b_! \F'' \to 0$. 
\end{proof}

\subsection{Semi-orthogonal decomposition} \label{subsec:sod}

\begin{prop} \label{prop:sod_main}
There are semi-orthogonal decompositions
\begin{equation} \label{eq:sod_main}
\DC^b(\Rcal) = \langle b^\ast \DC^b(\A), La^\ast \DC^b(\tilde{C}) \rangle, \quad \DC^-(\Rcal) = \langle b^\ast \DC^-(\A), La^\ast \DC^-(\tilde{C}) \rangle.
\end{equation}
\end{prop}

\begin{proof}
We will prove the result for $\DC^b(\Rcal)$; the other one is analogous. Given $M \in \DC^b(Y)$ and $N \in \DC^b(\tilde{C})$, we have $\Hom_\A(La^\ast M, b^\ast N) = \Hom_{\tilde{C}}(M, a_\ast b^\ast N) = \Hom_{\tilde{C}}(M, 0) = 0$ by direct computation. This shows the vanishing of Homs.

Now, let $\F \in \DC^b(\A)$ be given. Up to a quasi-isomorphism, we may replace $\F$ with a bounded complex $\E^\bullet$ with $\E^i = (E^i, \overline{E}^i; \phi^i) \in \Pscr$, thanks to property (2) in the proof of Lemma \ref{lemma:functors_Lbbounded}. Set $\G^i = (E^i, \rho_\ast(\mu_S^\ast(F|_S)); \id)$ and $\H^i = (0, \coker \phi^i; 0)$. We get short exact sequences $0 \to \G^i \to \E^i \to \H^i \to 0$. The differential of $\E^\bullet$ gives differentials on the $\G^i$ and $\H^i$, and so we get a short exact sequence of complexes:
\begin{equation} \label{eq:sod_ses}
0 \to \G^\bullet \to \E^\bullet \to \H^\bullet \to 0.
\end{equation}

To conclude, note that $\G^\bullet = La^\ast(E^\bullet)$ and $\H^\bullet = b^\ast (\coker \phi^\bullet)$. This way, the short exact sequence \eqref{eq:sod_ses} gives us a triangle $A \to \E^\bullet \to B \to A[1]$ with $B \in b^\ast \DC^b(\A)$ and $A \in La^\ast \DC^b(\tilde{C})$, as required.
\end{proof}

Various consequences can be obtained directly from the semi-orthogonal decomposition \eqref{eq:sod_main}. For instance, the unit $\id \to b^\ast Lb_!$ and the counit $La^\ast a_\ast \to \id$ give the triangles $La^\ast a_\ast(E) \to E \to b^\ast Lb_!(E) \to La^\ast a_\ast(E)[1]$ of the semi-orthogonal decomposition. Another important consequence is the fact that $\DC^b(\Rcal)$ is a \emph{geometric non-commutative scheme} in the sense of \cite{Orl16}*{Definition 4.3}.

\begin{cor} \label{cor:sod_ncsch}
The category $\DC^b(\Rcal)$ is equivalent to an admissible subcategory of $\DC^b(X)$, for some smooth, projective scheme $X$. 
\end{cor}

\begin{proof}
Follows directly from the semi-orthogonal decomposition of Proposition \ref{prop:sod_main} and by \cite{Orl16}*{Theorem 4.15}.
\end{proof}

\subsection{Categorical resolution} \label{subsec:catres}

In this subsection we will show that $\DC(\Rcal)$ is a categorical resolution of $\DC(C)$, in the sense of \cite{KL15}. To do so, we start by describing functors relating $\Coh(\Rcal)$ with $\Coh(C)$. We will use these to produce a semi-orthogonal decomposition at the unbounded level, which will be restricted to a Verdier quotient at the bounded level.

\begin{teo}
The (unbounded) derived category $\DC(\Rcal)$ is a categorical resolution of singularities of $\DC(C)$. This is, there exists functors $R\pi_\ast\colon \DC(\Rcal) \to \DC(C)$, $L\pi^\ast\colon \DC(C) \to \DC(\Rcal)$, satisfying the following.
\begin{enumerate}
\item $\DC(\Rcal)$ is a smooth, cocomplete, compactly generated triangulated category.
\item $R\pi_\ast \circ L\pi^\ast = \id$.
\item $L\pi^\ast$ and $R\pi_\ast$ commute with arbitrary direct sums.
\item $R\pi_\ast$ maps compact objects to $\DC^b(C)$.
\end{enumerate}
\end{teo}

We will omit the technical details of the proof, as they follow from the general construction of \cite{KL15}. Instead, we will focus on the construction of the functors $R\pi_\ast$ and $L\pi^\ast$.

First, set $\pi_\ast\colon \Coh(\Rcal) \to \Coh(C)$ as $\pi_\ast(F, \overline{F}; \phi) = \ker(\rho_\ast F \to \overline{F}_e)$, where the map is induced by $\phi$. Second, we define $\pi^\ast\colon \Coh(C) \to \Coh(\Rcal)$ via $\pi^\ast(M) = (\rho^\ast M, \coker(\mu_T^\ast(M|_T) \to\rho_\ast \mu_S^\ast(\rho^\ast M|_S)); \text{res})$, where the restriction map is induced by the unit $\id \to \rho_\ast \rho^\ast$.

\begin{lemma}
We have that $\pi^\ast \dashv \pi_\ast$. 
\end{lemma}

\begin{proof}
Let $M \in \Coh(C)$ and $(F, \overline{F}; \phi) \in \Coh(\A)$ be given. By definition, the set $\Hom_C(M, \pi^\ast(F, \overline{F}; \phi))$ is the same as morphisms $f\colon M \to \rho_\ast F$ such that the composition
\[ M|_T \xrightarrow{f|_T} (\rho_\ast F)|_T \to \overline{F}_e \]
is zero. By adjunction, this is the same as maps $g\colon \rho^\ast M \to F$ such that 
\[ M|_T \to \rho_\ast((\rho^\ast M)|_S) \to \rho_\ast (F|_S) \to \mu_{Z\ast}(\overline{F}) \]
is zero. Once again by adjunction, this is the same as a pair of morphisms $g\colon \rho^\ast M \to F$ and $h\colon \rho_\ast \mu_S^\ast((\rho^\ast M)|_S) \to \overline{F}$ such that the composition
\[ \mu_T^\ast (M|_T) \to \rho_\ast \mu_S^\ast((\rho^\ast M)|_S) \to \overline{F} \]
vanishes. This gives us the claimed description.
\end{proof}

\begin{lemma} \label{lemma:catres_ff}
The functors $\pi^\ast$ and $\pi_\ast$ induce derived functors $L\pi^\ast\colon \DC^-(C) \to \DC^-(\Rcal)$ and $R\pi_\ast\colon \DC^-(\Rcal) \to \DC^-(C)$. The functors $L\pi^\ast$ and $R\pi_\ast$ are adjoint, and the composition $R\pi_\ast \circ L\pi^\ast$ is isomorphic to the identity. 
\end{lemma}

\begin{proof}
First, note that $\pi^\ast$ preserves short exact sequences of locally free sheaves on $C$. In fact, note that $\rho^\ast$ is exact in short exact sequences of locally free sheaves. The exactness of the cokernels is a consequence of the snake lemma. This immediately gives us the existence of $L\pi^\ast$, thanks to Proposition \ref{prop:functors_existence}.

For the second part, consider the class $\Qscr \subseteq \Coh(\Rcal)$ consisting of objects $(F, \overline{F}; \phi)$ with $\rho_\ast F \to \overline{F}_e$ surjective. Let us show that $\Qscr$ satisfy the (dual of the) properties (1'), (2), (3) from Proposition \ref{prop:functors_existence}.
\begin{enumerate}
\item[(1')] Given $(F, \overline{F}; \phi) \in \Coh(\Rcal)$, we take some sheaf $E^0 \in \Coh(\tilde{C})$ with a surjection $E^0|_T \to \overline{F}_e$. This gives us a short exact sequence 
\[ 0 \to (F, \overline{F}; \phi) \to (F \oplus E^0, \overline{F}; \psi) \to (E^0, 0; 0) \to 0, \]
and so a resolution $0 \to \F \to \E^0 \to \E^1 \to 0$ with $\E^0, \E^1$ in $\Qscr$.
\item[(2)] This is a direct application of the snake lemma: take the diagram
\[ \begin{tikzcd}
0 \arrow[r] & \rho_\ast F \arrow[r] \arrow[d] & \rho_\ast F' \arrow[r] \arrow[d] & \rho_\ast F'' \arrow[r] \arrow[d] & 0 \\
0 \arrow[r] & \overline{F}_e \arrow[r] & \overline{F}'_e \arrow[r] & \overline{F}''_e \arrow[r] & 0,
\end{tikzcd} \]
and look at the induced sequence at the level of cokernels.

\item[(3)] Given an element $\F \in \Coh(\Rcal)$, consider the map $\psi\colon \rho_\ast F \to \overline{F}_e$ whose kernel corresponds to $\pi_\ast(F, \overline{F}; \phi)$. Our assumption ensures that the cokernel of this map is trivial.

This way, if $0 \to \F \to \F' \to \F'' \to 0$ is a short exact sequence with elements in $\Qscr$, the snake lemma applied to the maps $\psi, \psi', \psi''$ defined as above gives us the short exact sequence on the kernels. This is, a short exact sequence $0 \to \pi_\ast \F \to \pi_\ast \F' \to \pi_\ast \F'' \to 0$.
\end{enumerate}
This shows that $R\pi_\ast\colon \DC^b(\Rcal) \to \DC^b(C)$ is defined and preserves boundedness. Moreover, the functor $R\pi_\ast$ also exists as a functor from $\DC^-(\Rcal)$ to $\DC^-(C)$, cf. \cite{Stacks}*{Tag 07K7}. Note that the adjointness $L\pi^\ast \dashv R\pi_\ast$ follows immediately.

It remains to prove that $R\pi_\ast \circ L\pi^\ast(E) \cong E$, for which we may assume that $E$ is a bounded above complex of locally free sheaves. But in this case $R\pi_\ast \circ L\pi^\ast(E) = \pi_\ast \pi^\ast(E)$. The result follows from \eqref{eq:nrgen_keydiagram} after tensoring by $E$.
\end{proof}

From Lemma \ref{lemma:catres_ff} we have $\DC^-(\Rcal) = \langle \ker R\pi_\ast, L\pi^\ast \DC^-(C)\rangle$. In particular, the functor $R\pi_\ast\colon \DC^-(\Rcal) \to \DC^-(C)$ identifies $\DC^-(C)$ with the Verdier quotient $\DC^-(\Rcal)/\ker R\pi_\ast$. We will now adapt this result to the bounded derived categories.

\begin{lemma} \label{lemma:catres_esssurj}
The functor $R\pi_\ast$ preserves boundedness, and the restricted functor $R\pi_\ast\colon \DC^b(\Rcal) \to \DC^b(C)$ is essentially surjective.  
\end{lemma}

\begin{proof}
The first part follows by Proposition \ref{prop:functors_existence} together with the proof of Lemma \ref{lemma:catres_ff}. These two guarantee that $R^i\pi_\ast(E)=0$ for $i>2$ if $E \in \Coh(\Rcal)$. This way, if $F \in \DC^b(\Rcal)$ is given, the spectral sequence 
\begin{equation} \label{eq:catres_ss}
E_2 = R^p\pi_\ast(H^q(F)) \Rightarrow R^{p+q}\pi_\ast F
\end{equation}
converges at the $E_2$ page. 

Now, let $E \in \DC^b(C)$ be given, say $E \in \DC(C)^{[k, \ell]}$. By Lemma \ref{lemma:catres_ff}, we have that $E\cong R\pi_\ast F$ with $F = L\pi^\ast E \in \DC(\Rcal)^{(-\infty, \ell]}$. By the degeneration of the spectral sequence above we get that $E \cong R\pi_\ast \tau^{[k-1, \ell]}F$. 
\end{proof}

To further analyze the functor $R\pi_\ast\colon \DC^b(\Rcal) \to \DC^b(C)$ we follow the strategy in \cite{KKS22}*{p. 479}. Let us recall the key technical lemma.

\begin{lemma}[Verdier, \cite{Ver77}*{Théoremè 4--2, p. 20}] \label{lemma:catres_verdier}
Let $\D$ be a triangulated category, $\D'$ a full subcategory, and $\N$ a thick triangulated category. Assume that any of the following conditions hold:
\begin{enumerate}
\item[(b.ii)] For each $f\colon X \to Y$ with $X \in \D'$ and $Y \in \N$, we have that $f$ factors through an object of $\N \cap \D'$.
\item[(b.ii')] For each $f\colon Y \to X$ with $X \in \D'$ and $Y \in \N$, we have that $f$ factors through an object of $\N \cap \D'$.
\end{enumerate}
Then, the canonical functor $\D'/(\D' \cap \N) \to \D/\N$ is fully faithful. 
\end{lemma}

\begin{prop} \label{prop:catres_main}
The functor $R\pi_\ast\colon \DC^b(\Rcal) \to \DC^b(C)$ realizes $\DC^b(C)$ as a Verdier quotient of $\DC^b(\Rcal)$ by $\ker R\pi_\ast$.
\end{prop}

\begin{proof}
Let us verify Verdier's criterion to $\D'=\DC^b(\Rcal)$, $\D=\DC^-(\Rcal)$ and $\N=\ker R\pi_\ast$. To do so, consider a morphism $f\colon F \to G$ with $F \in \ker R\pi_\ast$ and $G \in \DC^b(\A)$. Given $N$, note that both $\tau^{\leq -N} F$ and $\tau^{>-N}F$ lie in $\ker R\pi_\ast$, by the degeneration of the spectral sequence \eqref{eq:catres_ss}. For $N \gg 0$, we have that $\Hom(\tau^{\leq -N}F, G)=0$, and so $f$ factors through $\tau^{>N}F \in \ker R\pi_\ast \cap \DC^b(\Rcal)$. This shows that $R\pi_\ast\colon \DC^b(\Rcal) \to \DC^b(C)$ is fully faithful; essential surjectivity follows from Lemma \ref{lemma:catres_esssurj}.
\end{proof}

\section{Stability conditions} \label{sec:stab}

In this section we will produce stability conditions on the category $\DC^b(\Rcal)$. To do so, we will start by describing the numerical Grothendieck group of $\DC^b(\Rcal)$ in Subsection \ref{subsec:Kgrp}. We will then produce a heart of a bounded t-structure on $\DC^b(\Rcal)$ in Subsection \ref{subsec:heart}, and a family of central charges on it in Subsection \ref{subsec:ccharge}.

The remaining work is to prove the support property. To do so, we will first give some constraints on the $\sigma_{\alpha, \beta, \gamma, \delta}$-semistable objects in Subsection \ref{subsec:sstable}. We will use this to prove various inequalities towards the support property in Subsection \ref{subsec:inequalities}. Finally, we will show the support property in Subsection \ref{subsec:quad}.

We work under the setup of Subsection \ref{subsec:setup}, but now assuming that $C$ is projective. This also guarantees that $\tilde{C}$ is projective.

\subsection{Grothendieck groups} \label{subsec:Kgrp}

Recall that for a triangulated category $\D$, we have the following two constructions. 
\begin{enumerate}
\item The \emph{Grothendieck group} $K(\D)$ is the abelian group generated by the objects of $\D$, subject to the relation $[B]=[A]+[C]$ for each triangle $A \to B \to C \to A[1]$.
\item Assume that $\D$ is $\C$-linear, with a Serre functor,  and that $\bigoplus_i \Hom(A, B[i])$ is a finite dimensional $\C$-vector space for any $A, B \in \D$. In this case we have the \emph{Euler pairing} $\chi(A, B) = \sum_i (-1)^i \dim \Hom(A, B[i])$. We set the \emph{numerical Grothendieck group} $K^{\num}(\D)$ to be quotient of $K(\D)$ by the kernel of the Euler pairing.
\end{enumerate}

\begin{prop}
The numerical Grothendieck group of $\DC^b(\Rcal)$ is free of rank $2n+me$. A basis of its dual is given by the maps that assign to $(F, \overline{F}; \phi)$ the following:
\begin{enumerate}
\item for each $1 \leq k \leq n$, the map $\deg_k(F, \overline{F}; \phi) = \deg(F|_{\tilde{C}_k})$;
\item for each $1 \leq k \leq n$, the map $\rk_k(F, \overline{F}; \phi) = \rk(F|_{\tilde{C}_k})$;
\item for each $1 \leq i \leq m$ and each $1 \leq j \leq e$, the map $\ell_{ij}(F, \overline{F}; \phi) = \ell((\overline{F}_j)|_{q_i})$.
\end{enumerate} 
\end{prop}

\begin{proof}
Recall from Proposition \ref{prop:sod_main} and Corollary \ref{cor:ausdc_sod} that we have semi-orthogonal decompositions
\[ \DC^b(\Rcal) = \langle b^\ast \DC^b(\A), La^\ast \DC^b(\tilde{C}) \rangle, \quad \DC^b(\A) = \langle \DC^b(T_0), \DC^b(T_0), \dots, \DC^b(T_0) \rangle. \]
Here $T_0=\Sing(C)$ is the set-theoretic support of $C$, and so $\DC^b(\A)$ admits a semi-orthogonal decomposition with $me$ factors isomorphic to $\DC^b(pt)$. Similarly, we have that $\DC^b(\tilde{C}) = \bigoplus_k \DC^b(\tilde{C}_k)$. This shows that
\[ K^{\num}(\Rcal) \cong \bigoplus_k K^{\num}(\tilde{C}_i) \oplus \bigoplus_{m,e} K^{\num}(\DC^b(pt)), \]
cf. \cite{AK25}*{Proposition 2.8(ii)}. This shows that $K^{\num}(\Rcal)$ is free of rank $2n+me$. 

For the second part, note that the maps $\deg_k, \rk_k, \ell_{ij}$ induce a linear map $K^{\num}(\Rcal) \to \Z^{2n+me}$. (The fact that these maps descend to $K^{\num}(\Rcal)$ follows immediately from the computations in \cite{AK25}*{Proposition 2.8(ii)}.) This way, it suffices to show that this map is surjective, which can be checked in appropriately chosen objects (such as $\O_{\tilde{C}_k}$, $\O_p$, and so on). 
\end{proof}

\subsection{Heart} \label{subsec:heart}


The next step towards constructing a pre-stability condition is to provide a heart on $\DC^b(\Rcal)$. To find it we will tilt the standard heart $\Coh(\Rcal)$. Set
\begin{align*}
\Ttors &= \{ (T, \overline{T}; \phi) : T \in \Coh(\tilde{C}), \overline{T} \in \Coh(\A), \phi \text{ surjective} \}, \\
\Ftors &= \{ (0, \overline{F}; 0) : \overline{F} \in \Coh(\A)\}.
\end{align*}

\begin{lemma}
The pair $(\Ttors, \Ftors)$ is a torsion pair on $\Coh(\A)$.
\end{lemma}

\begin{proof}
Given $\T = (T, \overline{T}; \phi) \in \Ttors$ and $\F=(0, \overline{F}; 0) \in \Ftors$, we have that $\Hom(\T, \F)$ consists on pairs $f\colon T \to 0$ and $g \colon \overline{T} \to \overline{F}$ making the diagram
\[ \begin{tikzcd}[column sep=huge] \rho_\ast \mu_S^\ast (T|_S) \arrow[r, "\rho_\ast\mu_S^\ast(f|_S)"] \arrow[d, "\phi"'] & 0 \arrow[d] \\ \overline{T} \arrow[r, "g"'] & \overline{F} \end{tikzcd} \]
commute. Here $f$ is the zero map, and the surjectivity of $\phi$ ensures that $g$ is the zero map as well. This proves that $\Hom(\Ttors, \Ftors)=0$.

Now, given $\E=(E, \overline{E}; \phi)$ an element of $\Coh(\A)$, it fits into the short exact sequence $0 \to (E, \im \phi; \phi) \to \E \to (0, \coker \phi; 0) \to 0$. The first term lies in $\Ttors$, while the second one is in $\Ftors$. This proves the claimed statement. 
\end{proof}

This way, we perform the tilt of the torsion pair to produce a heart
\[ \Cheart =\{ E \in \DC^b(\A) : H^{-1}(E) \in \Ftors, \ H^0(E) \in \Ttors, \ H^i(E)=0 \text{ otherwise} \}. \]

\begin{lemma} \label{lemma:heart_exactness}
The functor $R\pi_\ast\colon \DC^b(\Rcal) \to \DC^b(C)$ restricts to an exact functor $R\pi_\ast\colon \Cheart \to \Coh(C)$. Similarly, $a_\ast\colon \DC^b(\Rcal) \to \DC^b(\tilde{C})$ restricts to an exact functor $a_\ast\colon \Cheart \to \Coh(\tilde{C})$.
\end{lemma}

\begin{proof}
Given $E \in \Cheart$, we consider the exact triangle $H^{-1}(E)[1] \to E \to H^0(E)$. Here $R\pi_\ast(H^0(E)) = \pi_\ast(H^0(E))$ by the proof of Lemma \ref{lemma:catres_ff}. Also, note that $\pi_\ast(H^{-1}(E))=0$, and so $R\pi_\ast(H^{-1}(E)) = R^1\pi_\ast(H^{-1}(E))[-1]$. The first part follows immediately. A similar argument applies for the second part. 
\end{proof}

For later reference, we will need the following fact.

\begin{lemma} \label{lemma:heart_recollement}
Consider the exact triple $\DC^b(\A) \to \DC^b(\Rcal) \to \DC^b(\tilde{C})$ induced by the semi-orthogonal decomposition \eqref{eq:sod_main}. We have that the heart $\Cheart \subseteq \DC^b(\Rcal)$ is obtained by \emph{recollement} (cf. \cite{BBD82}*{Théoremè 1.4.10}) of the standard hearts $\Coh(\A)[1] \subseteq \DC^b(\A)$ and $\Coh(\tilde{C}) \subseteq \DC^b(\tilde{C})$.
\end{lemma}

\begin{proof}
It is clear that the functor $b^\ast\colon \DC^b(\A) \to \DC^b(\Rcal)$ maps $\Coh(\A)[1]$ into $\Cheart$, and that $a_\ast\colon \DC^b(\Rcal) \to \DC^b(\tilde{C})$ maps $\Cheart$ onto $\Coh(\tilde{C})$. This way, it suffices to show that $La^\ast \circ a_\ast$ is left t-exact, cf. \cite{BBD82}*{Proposition 1.4.12}. This follows directly by the fact that $a_\ast$ is t-exact, and that $La^\ast$ can be computed by resolving using locally free sheaves.
\end{proof}

\subsection{Central charge} \label{subsec:ccharge}

The next ingredient we need is a central charge. Pick $\alpha = (\alpha_k)_{1 \leq k \leq n}$, $\beta = (\beta_k)_{1 \leq k \leq n}$, $\gamma = (\gamma_k)_{1 \leq k \leq n}$ and $\delta = (\delta_{ij})_{1 \leq i \leq m, 1 \leq j \leq e}$ real numbers.

\begin{condition} \label{cond:ccharge_main}
For all $1 \leq k \leq n$, $1 \leq i \leq m$, $1 \leq j \leq e$, we impose the following:
\begin{enumerate}
\item $\alpha_k >0, \gamma_k >0$;
\item $\delta_{ij} >0$;
\item $\sum_j \ell(I^{e-j}|_{p}) \delta_{ij}<\gamma_k$ for all $p \in \tilde{C}_k \cap \rho^{-1}(q_i)$.
\end{enumerate}
\end{condition}

We point out that the condition is non-vacuous: take $\alpha_k=1$, $\beta_k=0$, $\gamma_k=1$, and $\delta_{ij}=\epsilon$ for some $0<\epsilon \ll 1$.

Under these assumptions, set $Z=Z_{\alpha, \beta, \gamma, \delta}\colon K^\num(\Rcal) \to \C$ via the formula
\[ Z(E) = \sum_{i=1}^m\sum_{j=1}^e \delta_{ij} \ell_{ij}(E) - \sum_{k=1}^n \gamma_k \deg_k(E) + \sum_{k=1}^n \beta_k \rk_k(E) + i \sum_{k=1}^n \alpha_k \rk_k(E). \]

\begin{lemma} \label{lemma:ccharge_upperhalf}
Let $0 \neq E \in \Cheart$ be given. Then $Z(E)$ lies in the upper half-space. 
\end{lemma}

\begin{proof}
Note that for any $E \in \Cheart$, we have that
\[ \rk_k(E) = \rk_k(H^0(E)) \geq 0, \qquad 1 \leq k \leq n. \]
Equality holds only if $a_\ast H^0(E)|_{\tilde{C}_k}$ is torsion. This immediately gives us that $\Im Z_{\alpha, \beta, \gamma, \delta}(E) \geq 0$, and equality holds when $a_\ast H^0(E)$ is torsion.

This way, assume now that $\Im Z_{\alpha, \beta, \gamma, \delta}(E) = 0$, so that $a_\ast H^0(E)$ is torsion. We will divide this case into four small steps. 
\begin{enumerate}
\item Assume that $E=H^{-1}(E)[1]$. In this case, we have that $\ell_{ij}(E)\leq 0$ for all $i,j$ and that the other terms are zero. This way, $\sum_{i,j} \delta_{ij} \ell_{ij}(E) \leq 0$. Equality holds only for $\ell_{ij}(E)=0$ for all $i,j$, i.e. for $E=0$.

\item Assume that $E=(\O_x, \overline{F}; \phi)$ for some $x \in \tilde{C}_k$. Our assumption ensures that $\phi$ is surjective. So, if $\rho(x) \notin \Sing(C)$, then $\overline{F}=0$, and $Z(E) = -\gamma_k<0$. Otherwise, we have that $\rho(x)=q_i$ for some $i$, and so 
\[ Z(E) = -\gamma_k + \sum_j \delta_{ij} \ell_{ij}(E) \leq -\gamma_k + \sum_j \delta_{ij} \ell(I^{e-j}|_{x}) <0, \]
thanks to the surjectivity. 

\item Assume that $E = H^0(E)$. Here $E=(F, \overline{F}; \phi)$ with $F$ torsion, so we induct in the length of $F$. The case $\ell(F)=1$ was done in the previous point. For the induction step, we pick a subsheaf $F' \subset F$ of length $1$, and use the subobject $(F', \phi(\mu_S^\ast(F'|_S)); \id)$ of $E$, which is a subobject both in $\Coh(\Rcal)$ and in $\Cheart$. The induction step allows us to conclude.

\item For the general case, note that $E$ fits in a short exact sequence $0 \to H^{-1}(E)[1] \to E \to H^0(E) \to 0$. We conclude directly by the previous points.
\end{enumerate}
This proves that $Z(E)$ lies in the upper half-space, as claimed.
\end{proof}

\begin{cor} \label{cor:ccharge_prestab}
Assume that $\alpha_k, \beta_k, \gamma_k, \delta_{ij}$ are all rational numbers satisfying Condition \ref{cond:ccharge_main}. Then, the pair $\sigma_{\alpha, \beta, \gamma, \delta}=(Z_{\alpha, \beta, \gamma, \delta}, \Cheart)$ is a pre-stability condition.
\end{cor}

\begin{proof}
Our assumption guarantees that $Z_{\alpha, \beta, \gamma, \delta}$ has discrete image. Moreover, the set $\{ E \in \Cheart: \Im Z(E)=0\}$ consists of objects $E$ with $H^0(E)=(F, \overline{F}; \phi)$ having $F$ torsion. This way, we immediately verify that the set is both Noetherian and Artinian. The existence of Harder--Narasimhan filtrations is now standard, cf. \cite{MS17}*{Lemma 4.10}.
\end{proof}

\subsection{Semistable objects} \label{subsec:sstable}

In this subsection we fix $\alpha_k, \beta_k, \gamma_k, \delta_{ij}$ rational numbers satisfying Condition \ref{cond:ccharge_main}. This way, we have a pre-stability condition $\sigma=\sigma_{\alpha, \beta, \gamma, \delta}$ as proven in Corollary \ref{cor:ccharge_prestab}. We start by describing the $\sigma$-stable objects of phase 1.

\begin{lemma} \label{lemma:sstable_phase1}
Let $E \in \Cheart$ be a $\sigma$-stable object of phase 1. Then $E$ satisfies one of the following three options.
\begin{enumerate}
\item We have $E=(\O_x, 0; 0)$ for some $x \in \tilde{C} \setminus S$. 
\item We have $E=(\O_p, \overline{F}; \phi)$, for some $p \in \rho^{-1}(q_i)$, and $\overline{F}$ is supported at $q_i$ with $\ell(\overline{F}_j) \leq \ell(I^{e-j}|_{q_i})$.
\item We have $E=(0, \overline{F}; 0)[1]$ for some $\overline{F}$.
\end{enumerate}
\end{lemma}

\begin{proof}
Follows from the proof of Lemma \ref{lemma:ccharge_upperhalf}.
\end{proof}

\begin{lemma} \label{lemma:sstable_phasenot1}
Let $E \in \Cheart$ be a $\sigma$-stable object of phase smaller than 1. Then $E=(F, \overline{F}; \phi)$ is an element of $\Coh(\Rcal)$, with $F$ a locally free sheaf.
\end{lemma}

\begin{proof}
For the first part, note that we have the short exact sequence
\[ 0 \to H^{-1}(E)[1] \to E \to H^0(E) \to 0 \]
in $\Cheart$. Here $H^{-1}(E)[1]$ must be zero, as otherwise it will be a subobject with phase 1. Let us write $E=(F, \overline{F}; \phi)$. 

We claim that $F$ must be locally free. Otherwise, the torsion subobject $F_{\mathrm{tors}}$ will induce a destabilizing subobject $(F_{\mathrm{tors}}, \phi(\mu_S^\ast(F_{\mathrm{tors}}|_S)); \phi)$. It follows that $F_{\mathrm{tors}}=0$, and so $F$ is locally free.
\end{proof}

\subsection{Inequalities} \label{subsec:inequalities}

The goal of this subsection is to produce various inequalities that will help us proving the support property. To do so, we endow $K^\num(\Rcal)$ with the norm $\norm{v}^2 = \sum_k \abs{\deg_k(v)}^2 + \sum_k \abs{\rk_k(v)}^2 + \sum_{i,j} \abs{\ell_{ij}(v)}^2$.

Let us start by bounding $\sigma$-stable objects $E \in \Cheart$ with phase 1. We omit the proof, as it follows by a direct computation together with Lemma \ref{lemma:sstable_phase1}.

\begin{lemma} \label{lemma:ineq_phase1}
Let $E \in \Cheart$ be a $\sigma$-stable object of phase 1.
\begin{enumerate}
\item If $E=(\O_x, 0; 0)$, for some $x \in \tilde{C}_k \setminus S$, then $\norm{E} = 1$ and $\abs{\Re Z(E)} = \abs{\gamma_k}$. 
\item If $E=(\O_{p}, \overline{F}; \phi)$ with $\rho(p) = q_i$ and $p \in \tilde{C}_k$, then
\[ \norm{E}^2\leq 1+\sum_j \ell(I^{e-j}|_{q_i})^2, \quad \abs{\Re Z(E)} \geq \gamma_k-\sum_j \ell(I^{e-j}|_{q_i})\delta_{ij}. \]

\item If $E=(0, \overline{F}; 0)[1]$, then
\[ \norm{E}^2=\sum_{ij} \ell(\overline{F}^j|_{p_i})^2, \quad \abs{\Re Z(E)} = \sum_{i,j}\delta_{ij} \ell(\overline{F}^j|_{p_i}). \]
\end{enumerate}
\end{lemma}

We will now focus on the case when $E \in \Cheart$ is a $\sigma$-stable object with phase smaller than 1. From Lemma \ref{lemma:sstable_phasenot1}, we have that $E=(F, \overline{F}; \phi)$ for some $F \in \Coh(\tilde{C})$ locally free. We write $F=\bigoplus_k F_k$, where each $F_k$ is a locally free sheaf on $\tilde{C}_k$.

\begin{lemma} \label{lemma:ineq_deg}
Let $E=(F, \overline{F}; \phi)$ be a $\sigma$-stable object of phase smaller than 1, and let $1 \leq k \leq n$. Then, the inequality
\[ \gamma_k \abs{\deg_k(E)} \leq \frac{\Im Z(E)}{\alpha_k} \left(\sum_{i,j} \delta_{ij}\ell(\rho_\ast(\O_S \otimes I^{e-j})) + \abs{\beta_k} \right) + \abs{\Re Z(E)} \]
holds.
\end{lemma}

\begin{proof}
Let us assume first that $\deg_k(E) \geq 0$. Consider the short exact sequence of sheaves $0 \to (F_k, \phi(\mu_S^\ast(F_k|_S)); \phi) \to E \to (F/F_k, \overline{F}/\phi(\mu_S^\ast(F_k|_S)); \phi) \to 0$, which is also a short exact sequence in $\Cheart$. The fact that $E$ is $\sigma$-stable gives us the inequality
\[ \frac{-\Re Z(F_k, \phi(\mu_S^\ast(F_k|_S)); \phi)}{\Im Z(F_k, \phi(\mu_S^\ast(F_k|_S)); \phi)} \leq \frac{-\Re Z(E)}{\Im Z(E)}. \]
Expanding the left hand side, and using that $F_k$ is supported on $\tilde{C}_k$, gives us
\[ \frac{-\sum_{i,j} \delta_{ij} \ell_{i,j}(F_k, \phi(\mu_S^\ast(F_k|_S)); \phi) + \gamma_k \deg_k(E) - \beta_k \rk_k(E)}{\alpha_k \rk_k(E)} \leq \frac{-\Re Z(E)}{\Im Z(E)}. \]
The surjectivity of $\phi$ ensures that
\[ \ell_{i,j}(F_k, \phi(\mu_S^\ast(F_k|_S)); \phi) \leq \ell(\rho_\ast(F_k|_S \otimes I^{e-j})) \leq \rk_k(F) \cdot \ell(\rho_\ast(\O_S \otimes I^{e-j})), \]
and so
\[ \frac{-\rk_k(E)\sum_{i,j} \delta_{ij} \ell(\rho_\ast(\O_S \otimes I^{e-j})) + \gamma_k \deg_k(E) - \beta_k \rk_k(E)}{\alpha_k \rk_k(E)} \leq \frac{-\Re Z(E)}{\Im Z(E)}. \]

We now rearrange terms and use that $\deg_k(E)\geq 0$ to get
\begin{align*}
&\, \gamma_k \deg_k(E) \\
\leq& \abs{-\Re Z(E) \cdot \frac{\alpha_k \rk_k(E)}{\Im Z(E)} + \rk_k(E)\sum_{i,j} \delta_{ij} \ell(\rho_\ast(\O_S \otimes I^{e-j})) + \beta_k \rk_k(E)} \\
\leq& \abs{\Re Z(E)} \cdot \abs{\frac{\alpha_k \rk_k(E)}{\Im Z(E)}} + \left(\sum_{i,j} \delta_{ij}\ell(\rho_\ast(\O_S \otimes I^{e-j})) + \abs{\beta_k} \right) \abs{\rk_k(E)}.
\end{align*}
We now use that $0 \leq \alpha_k \rk_k(E) \leq \Im Z(E)$ to conclude. The case $\deg_k(E) <0$ is handled similarly.
\end{proof}

\subsection{Quadratic forms} \label{subsec:quad}

In this subsection we will prove the full support property for the pre-stability conditions constructed in Corollary \ref{cor:ccharge_prestab}. Let us start with the following bound of semistable objects of phase 1.

\begin{lemma} \label{lemma:quad_bound1}
Let $\alpha_k, \beta_k, \gamma_k, \delta_{ij}$ rational numbers satisfying Condition \ref{cond:ccharge_main}. Let $E \in \Cheart$ be a $\sigma$-stable object of phase 1. Then, we have that $\norm{E}^2 \leq P_0 \abs{\Re Z(E)}^2$, where $P_0=P_0(\alpha_k, \beta_k, \gamma_k, \delta_{ij})$ is given by
\[ P_0 = \max \left\{ \left( \frac{1}{\gamma_k^2}\right)_k, \left( \frac{1+\sum_j \ell(I^{e-j}|_{q_j})^2}{(\gamma_k-\sum_j \ell(I^{e-j}|_{q_{ij}})\delta_{ij})^2}\right)_{i,k: q_i\in \rho(\tilde{C}_k)} , \left( \frac{1}{\delta_{ij}^2} \right)_i \right\}. \]
\end{lemma}

\begin{proof}
Follows immediately from the description in Lemma \ref{lemma:sstable_phase1}.
\end{proof}

The next step is the following inequality for semistable objects of phase smaller than 1.

\begin{lemma}
Let $\alpha_k, \beta_k, \gamma_k, \delta_i$ rational numbers satisfying Condition \ref{cond:ccharge_main}. If $E\in \Cheart$ is a $\sigma$-semistable object of phase smaller than 1, then $\norm{E}^2 \leq Q_1\abs{\Im Z(E)}^2+ P_1 \abs{\Re Z(E)}^2$, where
\begin{align*}
Q_1 =& m\sum_{j=1}^e \sum_{k=1}^n \left( \frac{\ell(\rho_\ast \O_{\tilde{C}_k \cap S} \otimes I^{e-j})}{\alpha_k}\right)^2 + \sum_{k=1}^n \frac{1}{\alpha_k^2} \\
& \quad + \sum_{k=1}^n \frac{2\left(\sum_{i,j} \delta_{ij}\ell(\rho_\ast(\O_S \otimes I^{e-j})) + \abs{\beta_k} \right)^2}{\alpha_k^2 \gamma_k^2}, \\
P_1=& \, 2\sum_{k=1}^n \frac{1}{\alpha_k^2}.
\end{align*}
\end{lemma}

\begin{proof}
Recall that $\norm{E}^2 = \sum_{k=1}^n \abs{\deg_k(E)}^2 + \sum_{k=1}^n \abs{\rk_k(E)}^2 + \sum_{i, j} \abs{\ell_{ij}}^2$. Let us bound each term independently.
\begin{itemize}
\item We have that $0 \leq \alpha_k \rk_k(E) \leq \Im Z(E)$. This gives us that $\abs{\rk_k(E)}^2 \leq \frac{1}{\alpha_k^2} \abs{\Im Z(E)}$. 
\item From Lemma \ref{lemma:ineq_deg}, we have that $\abs{\deg_k} \leq A \abs{\Re Z(E)} + B \abs{\Im Z(E)}$ for some functions $A, B$. We square both sides and use that $(x+y)^2 \leq 2x^2+2y^2$. 
\item From the proof of Lemma \ref{lemma:ineq_deg}, we have
\[ 0 \leq \ell_{ij} \leq \sum_k \rk_k(F) \cdot \ell(\rho_\ast \O_{\tilde{C}_k \cap S} \otimes I^{e-j}). \]
We square this expression and use the bound on the rank. 
\end{itemize}
Putting everything together gives us the claimed inequality.
\end{proof}

\begin{cor} \label{cor:quad_stab}
Let $\Omega \subset \R^{3n+me}$ be the open subset of tuples $(\alpha_k, \beta_k, \gamma_k, \delta_{ij})$ satisfying Condition \ref{cond:ccharge_main}. There exists continuous functions $P, Q\colon \Omega \to \R_{>0}$ such that the following holds: if $E \in \Cheart$ is a $\sigma_{\alpha, \beta, \gamma, \delta}$-stable object, with $\alpha, \beta, \gamma, \delta$ rational, then
\[ \norm{E}^2 \leq P(\alpha,\beta,\gamma,\delta) \abs{\Re Z_{\alpha,\beta, \gamma, \delta}(E)}^2 + Q(\alpha,\beta,\gamma,\delta) \abs{\Re Z_{\alpha,\beta, \gamma, \delta}(E)}^2. \]
In particular, we get stability conditions $\sigma=\sigma_{\alpha, \beta, \gamma, \delta}$ for all $(\alpha, \beta, \gamma, \delta) \in \Omega$.
\end{cor}

\begin{proof}
The first part is immediately from the previous computation, taking $P=\max\{P_0, P_1\}$ and $Q=Q_1$. If $\alpha, \beta, \gamma, \delta$ are all rational, this shows that the pre-stability condition $\sigma_{\alpha, \beta, \gamma, \delta}$ from Corollary \ref{cor:ccharge_prestab} satisfies the support property. The extension to $\alpha, \beta, \gamma, \delta$ real is standard, cf. \cite{Tod14}*{\textsection 5.1}. 
\end{proof}

\section{Existence of moduli spaces} \label{sec:moduli}

So far, we have constructed a categorical resolution of singularities $\DC^b(\Rcal)$ of $\DC^b(C)$, and stability conditions $\sigma = \sigma_{\alpha, \beta, \gamma, \delta}$. In this section we will investigate the existence of moduli spaces of $\sigma$-semistable objects with a fixed numerical vector $v$. 

Let us fix some notation. By Corollary \ref{cor:sod_ncsch}, there exists a smooth, proper variety $X$, together with a fully faithful functor $\DC^b(\Rcal) \to \DC^b(X)$ with left and right adjoints. This way, if $\sigma$ is one of the stability conditions of Corollary \ref{cor:quad_stab}, and $v \in K^\num(\Rcal)$ are given, we consider the assignment $\M_\sigma(v)$ as follows. For each $\C$-scheme $S$, set $\M_\sigma(v)(S)$ to be
\[ \{E \in \DC_{S\dashperf}(S \times X) : E|_s \in \Cheart \text{ is }\sigma\text{-semistable of class }v\text{, for all }s \in S\} \]
This defines a subfunctor of the stack $\M_{\mathrm{pug}}(X)$ of perfect, universally gluable objects from \cite{Lie06}. The goal of this section is to prove the following result.

\begin{teo} \label{teo:moduli_main}
For each $v \in K^\num(\Rcal)$, the assignment $\M_\sigma(v)$ is a finite type algebraic stack admitting a proper good moduli space $M_\sigma(v)$. 
\end{teo}

\begin{proof}
We follow the discussion from \cite{AHLH23}*{Example 7.29} (cf. \citelist{\cite{BM23}*{\textsection 2.1} \cite{AS25}*{\textsection 2.10}}). To prove the claimed result, it suffices to show that (i) for each $v \in K^{\num}(\Rcal)$, the collection of $\sigma$-semistable objects with class $v$ is \emph{bounded}; and (ii) the heart $\Cheart \subseteq \DC^b(\Rcal)$ satisfies \emph{generic flatness}.

We will prove boundedness in Subsection \ref{subsec:boundeness}, and generic flatness in Subsection \ref{subsec:genflat}. This way, the discussion in \cite{AHLH23}*{Example 7.29} will immediately imply the claimed result.
\end{proof}

\begin{remark}
We will prove later that for nodes and tacnodes the moduli spaces $M_\sigma(v)$ are actually \emph{projective}, cf. Theorems \ref{teo:node_main} and \ref{teo:tac_main}. See also the discussion in \cite{BM23}*{Theorem 2.4}.
\end{remark}

\subsection{Boundedness} \label{subsec:boundeness}

Recall that a collection of objects $\Scal$ in $\DC^b(X)$ is said to be \emph{bounded} if there exists a scheme $T$ and a family $\E \in \DC_{T\dashperf}(T \times X)$ such that each $F \in \Scal$ is isomorphic to $\E|_t$ for some $t \in T$. The goal of this subsection is to prove the following result.

\begin{prop} \label{prop:boundedness_main}
Let $\sigma = \sigma_{\alpha, \beta, \gamma, \delta}$ and $v \in K^\num(X)$ be given. The collection of objects in $\Cheart$ that are $\sigma$-semistable with numerical vector $v$ is bounded.
\end{prop}

The proof will rely in the following two tools. First, we will use the well-known fact that on a curve, the collection of semistable vector bundles of fixed rank and degree is bounded.

\begin{lemma}[cf. \cite{HL10}*{Theorem 3.3.7}] \label{lemma:boundedness_curves}
Let $\Gamma$ be a smooth, projective curve; let $r>0$, and let $d \in \Z$ be an integer. Then, the collection of locally free sheaves in $\Gamma$ that are slope-semistable of rank $r$ and degree $d$ is bounded. 
\end{lemma}

The second fact we will use is that boundedness is preserved under extensions.

\begin{lemma}[\cite{Tod08}*{Lemma 3.16}] \label{lemma:boundedness_extensions}
Let $\Scal_1, \Scal_2, \Scal_3$ be three collections of objects in $\DC^b(X)$. Assume that $\Scal_1$ and $\Scal_2$ are bounded. Suppose that each object $E \in \Scal_3$ fits into a distinguished triangle $A \to E \to B \to A[1]$, with $A \in \Scal_1$ and $B \in \Scal_2$. Then $\Scal_3$ is bounded as well.
\end{lemma}

We are ready to start with the proof of Proposition \ref{prop:boundedness_main}. Recall that any object $E \in \Cheart$ has $\rk_k(E) \geq 0$ (cf. the proof of Lemma \ref{lemma:ccharge_upperhalf}). This way, we assume that $\rk_k(v) \geq 0$ for all $1 \leq k \leq n$. We will divide the proof in two cases, depending on whether $\rk_k(v)=0$ for all $k$ or not.

\begin{proof}[Proof of Proposition \ref{prop:boundedness_main}, $\rk_k(v)=0$ for all $k$] 
Assume that $v$ satisfies $\rk_k(v)=0$ for all $1 \leq k\leq n$. This way, we have that $\Im Z(v)=0$, and so the objects with numerical class $v$ all have phase 1. 

Now, we classified the \emph{stable} objects of phase $1$ in Lemma \ref{lemma:sstable_phase1}. For a fixed $v$, each of the three families described there are bounded:
\begin{enumerate}
\item The collection of objects $\{ (\O_x, 0; 0) \}_{x \in \tilde{C} \setminus Z}$ is bounded: take $\tilde{C} \setminus Z$, and the structure sheaf of the diagonal.

\item The collection $\{ (\O_q, \overline{F}; \phi) \}$ is bounded. In fact, we have finitely many options for $p$, while the $\ell(\overline{F}^j)$ are determined by $v$. Once those are fixed, the object $\overline{F}$ is determined as a vector space; the module structure and the maps $\alpha, \beta$ give a finite dimensional vector space, and the compatibility conditions cut out a closed subscheme on it. 

\item The collection $\{ (0, \overline{F}; 0)[1] \}$ is bounded, by an argument similar to the one in the previous point.
\end{enumerate}
Lastly, note that there are finitely many decompositions of $v$ as a sum of numerical classes in (1)--(3). An application of Lemma \ref{lemma:boundedness_extensions} ensures that the collection $\{ E \in \Cheart: v(E)=v, E \text{ semistable}\}$ is finite. 
\end{proof}

\begin{proof}[Proof of Proposition \ref{prop:boundedness_main}, $\rk_k(v)>0$ for some $k$] Assume that $v$ satisfies $\rk_k(v)>0$ for some $k$. This way, we have that $\Im Z(v)>0$, so that the objects with numerical class $v$ have phase smaller than 1.

Let $E$ be one of such objects. By Lemma \ref{lemma:sstable_phasenot1}, we have that $E=(F, \overline{F}; \phi)$ is an element of $\Coh(\A)$, with $F$ a locally free sheaf on $\tilde{C}$. Write $F = \bigoplus_k F_k$, where each $F_k$ is a locally free sheaf supported on $\tilde{C}_k$.

We claim that for each $k$, and for each quotient of sheaves $F_k \to G \to 0$, the slope $\deg(G)/\rk(G)$ is bounded below. To show this, note that we have a quotient in $\Cheart$: $(F, \overline{F}; \phi) \to (G, \phi(G); \phi) \to 0$. Using that $(F, \overline{F}; \phi)$ is $\sigma$-semistable yields
\[ -\frac{\Im Z(v)}{\Re Z(v)} \leq -\frac{\Im Z(G, \phi(G); \phi)}{\Re Z(G, \phi(G); \phi)} \leq \frac{\gamma_k}{\alpha_k} \cdot \frac{\deg(G)}{\rk(G)} - \frac{\beta_k}{\alpha_k}. \]
A similar argument shows that for each $k$, and each subsheaf $0 \to H \to F_k$, the slope $\deg(H)/\rk(H)$ is bounded above. 

From here, it follows that there are finitely many Harder--Narasimhan filtrations for each $F_k$ (with respect to the slope $\deg/\rk$). Using Lemma \ref{lemma:boundedness_curves} and Lemma \ref{lemma:boundedness_extensions}, we get that the collection
\[ \{ F \in \Coh(\tilde{C}): \exists E \in \Cheart, v(E)=v, \text{ semistable}, E=(F, \overline{F}; \phi) \} \]
is bounded. Using that the collection $\{(0, \overline{F}; 0)[1]\}$ with lengths determined by $v$ is bounded, another application of Lemma \ref{lemma:boundedness_extensions} ensures that the collection of $\sigma$-semistable objects with numerical vector $v$ is bounded.
\end{proof}

\subsection{Generic flatness} \label{subsec:genflat}

Let us start by recalling some definitions. Given an admissible subcategory $\D \subseteq \DC^b(X)$, where $X$ is a smooth, projective variety, we have a semi-orthogonal decomposition $\DC^b(X) = \langle \D, {}^\perp \D\rangle$. This way, the main result of \cite{Kuz11} gives us a base change $\DC^b(S \times X) = \langle \D_S {}^\perp \D_S \rangle$. If $\H \subset \D$ is the heart of a bounded t-structure on $\D$, by \cite{Pol07} (cf. \cite{AP06}) there is a heart $\H_S \subset \D_S$.

Following \cite{AP06}*{Problem 3.5.1}, we say that $\H$ satisfies \emph{generic flatness} if for any $S$ and any $E \in \H_S$, there exists an open dense subset $U \subset S$ such that $E|_s \in \H$ for each $s \in U$. For the standard heart of $\DC^b(X)$, this holds by the usual version of generic flatness. The goal of this subsection is to prove the following statement.

\begin{prop} \label{prop:genflat_main}
The heart $\Cheart \subseteq \DC^b(\Rcal)$ satisfies generic flatness.
\end{prop}

The key ingredient is the following proposition, which says that generic flatness is preserved under recollement.

\begin{lemma}[\cite{BRH21}*{Lemma 4.27}] \label{lemma:genflat_recollement}
Let $\D \subseteq \DC^b(X)$ be an admissible subcategory. Assume that $\D$ admits a semi-orthogonal decomposition $\D=\langle \D_1, \D_2\rangle$, and fix hearts of bounded t-structures $\H_i \subseteq \D_i$. If each $\H_i$ satisfies generic flatness, and if $\H \subset \D$ is obtained by recollement, then $\H$ satisfies generic flatness.
\end{lemma}

\begin{proof}[Proof of Proposition \ref{prop:genflat_main}]
Recall from \eqref{eq:sod_main} that $\DC^b(\Rcal)$ admits a semi-orthogonal decomposition $\DC^b(\A) = \langle b^\ast \DC^b(\A), La^\ast \DC^b(\tilde{C}) \rangle$. We also proved in Lemma \ref{lemma:heart_recollement} that $\Cheart$ is a recollement of the hearts $\Coh(\A)[1] \subseteq \DC^b(\A)$ and $\Coh(\tilde{C}) \subseteq \DC^b(\tilde{C})$. By Lemma \ref{lemma:genflat_recollement} it suffices to show that $\Coh(\A) \subseteq \DC^b(\A)$ satisfies generic flatness.

To show that $\Coh(\A_{T, I, e}) \subseteq \DC^b(\A_{T, I, e})$ satisfies generic flatness, we proceed by induction on $e$. The case $e=0$ is immediate, as here $T$ is smooth, and $\A_{T, I, 0} = \Coh(T)$. For the induction step, we use the exact triple
\[ \DC^b(Z_0) \xrightarrow{i_\ast} \DC^b(\A_{Z, I, n}) \xrightarrow{e^!} \DC^b(\A_{Z', I, n-1}). \]
Here $e_\ast e^!\colon \DC^b(\A_{Z, I, n}) \to \DC^b(\A_{Z, I, n})$ is clearly t-exact with respect to $\Coh(\A_{Z, I, n})$. It follows that $\Coh(\A_{Z, I, n})$ is a recollement of $\Coh(Z_0)$ and $\Coh(\A_{Z', I, n-1})$ by \cite{BBD82}*{Proposition 1.4.12(i)}. The induction assumption, together with Lemma \ref{lemma:genflat_recollement}, shows that $\Coh(\A_{Z, I, n})$ satisfies generic flatness.
\end{proof}


\section{Comparison} \label{sec:comparison}

Let us assume from now on that $C$ is an irreducible, reduced curve. We follow the notation of Subsection \ref{subsec:setup}; in particular, we are fixing a non-rational locus $T$ of $C$, and an exponent $e\geq 1$. 

Let us specialize the results of Theorem \ref{teo:intro_main} to our setting. We have constructed an abelian category $\Coh(\Rcal)$, a heart $\Cheart \subset \DC^b(\Rcal)$, and stability conditions $\sigma$ depending on $3+me$ parameters. We fix $\alpha=\gamma=1$ and $\beta=0$, and let
\begin{equation} \label{eq:comparison_region}
\tilde{\Omega} = \left\{ (\delta_{ij})_{1 \leq i \leq m, 1 \leq j \leq e} : \delta_{ij}>0, \ \sum_j \ell(I^{e-j}|_{p}) \delta_{ij}<1 \forall p \in \Coh(\tilde{C}) \right\} \subset \R^{me}.
\end{equation}
This way, Theorem \ref{teo:intro_main} gives us stability conditions $\sigma_\Delta$ with central charge
\[ Z_{\Delta}(E) = \sum_{i=1}^m \sum_{j=1}^e \delta_{ij} \ell_{ij}(E) - \deg(E) + i \rk(E) \]
for any $\Delta = (\delta_{ij}) \in \tilde{\Omega}$, varying continuously on $\Delta$.

\begin{remark}
Note that \eqref{eq:comparison_region} is non-empty: taking $0<\delta_{ij} \ll 1$ satisfies it. Also, taking $0 <1-\delta_{ie} \ll 1$ and $0 < \delta_{ij} \ll 1$ for all $j<e$ also lies in \eqref{eq:comparison_region}.
\end{remark}

We now fix a numerical vector $v \in K^{\num}(\Rcal)$. We have a \emph{wall-and-chamber} decomposition on $\Stab(\Rcal)$: a collection of locally finite, locally closed submanifolds of real codimension 1, called \emph{walls}, dividing $\Stab(\Rcal)$ into \emph{chambers}. By pulling back these walls to $\tilde{\Omega}$, we get an induced wall-and-chamber decomposition.

The goal of this section is twofold. First, we will prove finiteness of the walls with respect to $v$ on $\tilde{\Omega}$. Second, we will use this to compare the moduli spaces $M_\sigma(v)$ with moduli spaces of slope-semistable sheaves on $\tilde{C}$ and $C$.

\begin{teo} \label{teo:comparison_main}
Let $v = (r, d, (\ell_{ij})_{ij})$ be a numerical vector with $r>0$. Assume that $v$ is primitive. 
\begin{enumerate}
\item There are finitely many walls for $v$ in $\tilde{\Omega}$, and each one of them defines a (non-zero) hyperplane in $\tilde{\Omega}$. In particular, for $\sigma \in \tilde{\Omega}$ outside of these hyperplanes, we have that an object $E$ with numerical class $v$ is $\sigma$-semistable if and only if it is $\sigma$-stable.
\item There exists $\epsilon>0$ such that the functor $a_\ast\colon \DC^b(\Rcal) \to \DC^b(\tilde{C})$ induces a proper map $a_\ast\colon M_{\sigma}(v) \to M_{\tilde{C}}(r, d)$ for any $\Delta \in \tilde{\Omega}$ with $\delta_{ij}<\epsilon$ for all $i,j$. 
\item There exists $\epsilon'>0$ such that the functor $R\pi_\ast\colon \DC^b(\Rcal) \to \DC^b(C)$ induces a proper map $R\pi_\ast\colon M_{\sigma}(v) \to M_C(r, \overline{d})$ for any $\Delta \in \tilde{\Omega}$ with $\delta_{ie}>1-\epsilon'$ for all $i$, where $\overline{d}=d+r(p_a(C)-g)-\sum_i \ell_{ie}$. 

\item Assume that $v=v(\pi^\ast E)$ for some $E \in M_C(r, \overline{d})$. If $R\pi_\ast\colon M_\sigma(v) \to M_C(r, \overline{d})$ is the map from (3), we have that the map is an isomorphism over the locus of $M_C(r, \overline{d})$ parametrizing vector bundles. 
\end{enumerate}
\end{teo}

In general, it is difficult to give a precise description of the maps from Theorem \ref{teo:comparison_main}. As we will see in Section \ref{sec:node}, the map $R\pi_\ast$ is in general not surjective. In special cases however, one can make the maps $a_\ast$ and $R\pi_\ast$ more explicit. We will return to this point in later sections.

The proof of Theorem \ref{teo:comparison_main} relies on a basic description of the $\sigma$-semistable objects of phase smaller than 1. We will analyze these objects in Subsection \ref{subsec:ssnot1}, and will use this to prove Theorem \ref{teo:comparison_main} in Subsection \ref{subsec:pfcomparison}.

\subsection{Semistable objects of phase not 1} \label{subsec:ssnot1}

The goal of this subsection is to give a basic description of the $\sigma$-semistable objects of phase smaller than 1. We fix $\Delta \in \tilde{\Omega}$, and $v$ a numerical class with $\rk(v)>0$. We recall from Lemma \ref{lemma:sstable_phasenot1} that if $E\in \Cheart$ is $\sigma$-semistable with class $v$, then $H^{-1}(E)=0$, and $E=(F, \overline{F}; \phi)$ with $F$ locally free. As a consequence, we get the following observation.

\begin{cor} \label{cor:ssnot1_smallrk}
Assume that $E=(F, \overline{F}; \phi)$ is strictly semistable, fitting into a exact sequence $0 \to E' \to E \to E'' \to 0$ in $\Cheart$ with $E', E''$ semistable of the same phase. Then $\rk(E'), \rk(E'')<\rk(E)$.
\end{cor}

\begin{proof}
Note that $E'=(F', \overline{F'}; \phi')$ and $E''=(F', \overline{F}''; \phi'')$. This way, the short exact sequence in $\Cheart$ is also a short exact sequence in $\Coh(\Rcal)$, and it induces a short exact sequence $0 \to F' \to F \to F'' \to 0$ in $\Coh(\tilde{C})$. In particular, this shows that $\rk(F')+\rk(F'')=\rk(F)$. As $F', F''$ must have positive rank, we get the claimed inequality.
\end{proof}

\subsection{Proof of Theorem \ref*{teo:comparison_main}} \label{subsec:pfcomparison}

In this subsection we will prove Theorem \ref{teo:comparison_main}. We start by proving the first statement, for which we need to describe the walls associated to $v \in K^{\num}(\Rcal)$.

Let us recall that a \emph{numerical wall} corresponding to a decomposition $v=v'+v''$ is given by the preimage of the set
\[ \{ W \in \Hom_\Z(K^{\num}(\Rcal), \C) : \Re W(v) \cdot \Im W(v') = \Re W(v') \cdot \Im W(v) \} \]
under the forgetful map $\Stab(\Rcal) \to \Hom_\Z(K^\num(\Rcal), \C)$. This way, the preimage in $\tilde{\Omega}$ is given by $\Delta \in \tilde{\Omega}$ such that $\Re Z_\Delta(v) \cdot \Im Z_\Delta(v') = \Re Z_\Delta(v') \cdot \Im Z_\Delta(v)$.

Let us spell this out. Write $v=(r, d, \ell_{ij})$ and $v' = (r', d', \ell_{ij})$, so that the numerical wall is given by
\begin{equation} \label{eq:pfcomparison_wallformula}
r\left( \sum_{i,j} \delta_{ij} \ell_{ij}' - d' \right) = r' \left( \sum_{i,j} \delta_{ij} \ell_{ij} - d \right).
\end{equation}
Note that by Corollary \ref{cor:ssnot1_smallrk}, it suffices to look at those walls with $0<r'<r$.

\begin{claim}
If $v$ is primitive, then the equation \eqref{eq:pfcomparison_wallformula} is not zero. Otherwise, we will have that $d'/r'=d/r$ and $\ell_{ij}'/r'=\ell_{ij}/r$ for all $i,j$. But by assumption we have that there are integers $a, b, c_{ij}$ such that $ar+bd+\sum_{ij}c_{ij}\ell_{ij}=1$. This way, we get that $ar'+bd'+\sum_{ij}c_{ij}\ell_{ij}'=r'/r$, a contradiction with $r'<r$.
\end{claim}

We can rewrite \eqref{eq:pfcomparison_wallformula} as $\sum_{ij} \left( \ell_{ij}/r - \ell_{ij}'/r' \right) \delta_{ij} = d/r - d'/r'$. Given $0 <r' < r$, there are finitely many choices for $\ell_{ij}'$, as we need surjectivity of $\phi$. For each one of these choices, there are finitely many options of $d$ such that the hyperplane has a non-empty intersection with $\tilde{\Omega}$, as $\tilde{\Omega}$ is bounded. This finishes the proof of the first part in Theorem \ref{teo:comparison_main}.

Let us turn our attention to the second part of Theorem \ref{teo:comparison_main}. Fix a numerical class $v=(r, d, \ell_{ij})$ with $r>0$. Set $\epsilon = (r^2 \cdot \sum_{ij}\ell(I^{e-j}|_{q_i}))^{-1}$.

\begin{lemma} \label{lemma:pfcomparison_pushforward}
Let $\Delta \in \tilde{\Omega}$ be a class with $\delta_{ij} < \epsilon$ for all $i, j$. If $E=(F, \overline{F}; \phi)$ is a $\sigma$-semistable object of phase smaller than 1, then $F \in \Coh(\tilde{C})$ is slope-semistable.
\end{lemma}

\begin{proof}
We argue by contradiction. Assume that $F$ is not slope-semistable, and let $F' \subset F$ be a destabilizing vector bundle. This induces a subobject $E'=(F', \phi(\mu_S^\ast(F'|_S)); \phi)$ of $E$. Denote by $(r', d', \ell_{ij}')$ the numerical vector of $E'$

As $E$ is $\sigma$-semistable, we have that
\[ \frac{d-\sum_{ij}\delta_{ij}\ell_{ij}}{r} \geq \frac{d'-\sum_{ij} \delta_{ij}\ell_{ij}'}{r'}. \]
By the surjectivity, we have that $\ell_{ij}' \leq r' \ell(I^{e-j}|_{q_i})$. This, together with our choice of $\epsilon$, gives us
\[ \frac{d}{r} \geq \frac{d'-\sum_{ij} \delta_{ij}\ell_{ij}'}{r'} \geq \frac{d'}{r'} - \epsilon \frac{r'}{r} \sum_{ij} \ell(I^{e-j}|_{q_i}) > \frac{d'}{r'} - \frac{1}{r!}. \]
This shows that $1/r^2 > d'/r' - d/r >0$, a contradiction.
\end{proof}

This lemma shows that the functor $a_\ast\colon \DC^b(\Rcal) \to \DC^b(\tilde{C})$ maps $\sigma$-semistable objects of phase $v$ to slope-semistable sheaves of rank $r$ and degree $d$, provided that $\delta_{ij}<\epsilon$ for all $i, j$. In particular, we get a well-defined map $a_\ast\colon \M_\sigma(v) \to \M_{\tilde{C}}(r, d)$ at the level of stacks. This immediately descends to a morphism $a_\ast\colon M_\sigma(v) \to M_{\tilde{C}}(r, d)$ of the corresponding good moduli spaces (e.g. by \cite{Alp13}*{Theorem 6.6}). This mas is proper, as both sides are proper. This proves the second part of Theorem \ref{teo:comparison_main}.

The proof of the last part of Theorem \ref{teo:comparison_main} follows the same ideas. We fix $v=(r, d, \ell_{ij})$ with $r>0$, and set $\epsilon'=(2r^2 \sum_{ij}(I^{e-j}|_{q_i}))^{-1}$. If $E$ is a $\sigma$-semistable object with numerical vector $v$, then $E=(F, \overline{F}; \phi)$ with $E$ locally free. By construction, we have
\[ R\pi_\ast(E) = \pi_\ast(E) = \ker(\rho_\ast F \to \overline{F}_e), \]
cf. Subsection \ref{subsec:catres} (and the proof of Lemma \ref{lemma:catres_ff}). In particular, we have that $\pi_\ast(E)$ has rank $r$ and degree $\overline{d} := d+r(p_a(C)-g)-\sum_i \ell_{ie}$.

\begin{lemma}
Assume that $\delta_{ie}>1-\epsilon'$ for all $i$. Then, we have that $R\pi_\ast(E)$ is a slope-semistable pure dimension 1 sheaf in $C$.
\end{lemma}

\begin{proof}
As before, we argue by contradiction. We assume that $R\pi_\ast(E) = \pi_\ast E$ is not slope-semistable, and we let $G \in \Coh(C)$ be a saturated, destabilizing subsheaf.

By adjunction, the inclusion $G \to \pi_\ast E$ corresponds to $\pi^\ast G \to E$. Now, the left hand side is of the form $\pi^\ast G = (H, \overline{H}; \eta)$, where $H$ might have torsion (over $S$). We let $B'=H_{tf}$ the quotient of $H$ by its torsion subsheaf, and $\overline{B}' = \eta(\mu_S^\ast(H_{tf}|_S))$. We point out that the map $\pi^\ast G \to E$ factors through $A = (B', \overline{B}'; \eta)$. Moreover, the induced map $A \to E$ is a monomorphism in the category $\Cheart$. This way, if $v(A)=(r', d', \ell_{ij}')$, then we get the inequality
\[ \frac{d-\sum_{ij}\delta_{ij}\ell_{ij}}{r} \geq \frac{d'-\sum_{ij} \delta_{ij}\ell_{ij}'}{r'}. \]
Our choice of $\epsilon'$ allows us to bound the left hand side as
\[ \geq \frac{d'-\sum_i \ell_{ie}'}{r'} - \frac{\sum_{i} \sum_{j < e}\delta_{ij}\ell_{ij}'}{r'} \geq \frac{d'-\sum_i \ell_{ie}'}{r'} -\epsilon'. \]
Similarly, the left hand side can be bounded above by $(d-\sum_i \ell_{ie})/r + \epsilon'$. 

Finally, note that the natural map $G \to \pi_\ast A$ is generically an isomorphism; thus, the cokernel is supported at points. In particular, it has to be an isomorphism, as $G \to \pi_\ast A \to \pi_\ast E$ is a saturated subsheaf. We conclude as above, noting that $\deg G=d'-\sum_i \ell_{ie}'+r(p_a(C)-g)$ and similar for $\pi_\ast A$.
\end{proof}

As in the second part, this shows that $R\pi_\ast$ induces a morphism $R\pi_\ast\colon \M_\sigma(v) \to \M_C(r, \overline{d})$, and so a morphism at the level of good moduli spaces. To prove the last statement of Theorem \ref{teo:comparison_main}, we need the following lemma.

\begin{lemma} \label{lemma:pfcomparison_pullback}
If $E \in M_C(r, \overline{d})$, then $\pi^\ast E$ is $\sigma$-semistable.
\end{lemma}

\begin{proof}
We argue by contradiction; so that there exists a short exact sequence $0 \to \F \to \pi^\ast E \to \G \to 0$ in $\Cheart$ destabilizing $E$. As before, we may assume that $\F, \G$ lie in $\Coh(\Rcal)$, and that the corresponding component in $\Coh(\tilde{C})$ is locally free. This way, we apply $R\pi_\ast$ by directly evaluating $\pi_\ast$, cf. the proof of Lemma \ref{lemma:catres_ff}. We get the short exact sequence
\[ 0 \to \pi_\ast \F \to E \to \pi_\ast G \to 0 \]
in $\Coh(C)$. 

Now, the fact that $E$ is slope-semistable gives us the inequality
\[ \frac{\deg(\pi_\ast \F)}{\rk(\pi_\ast \F)} < \frac{\deg E}{\rk E} \implies \frac{\deg (\pi_\ast \F)}{\rk(\pi_\ast \F)} + \frac{1}{r^2} \leq \frac{\deg E}{\rk E}. \]
We now use that $\deg(\pi_\ast \F) = \deg F + \rk(F)(p_a(C)-g) - \sum_i \ell_{ie}(F)$, together with the fact that $\F$ destabilizes $\pi^\ast E$, to reach a contradiction. 
\end{proof}

\begin{proof}[Proof of Theorem \ref{teo:comparison_main}, (4)]
We let $U \subset M_C(r, \overline{d})$ be the open subset parametrizing vector bundles. We have a twisted universal family $\E \in \DC^b(U \times C, \alpha \boxtimes 1)$, for some Brauer class $\alpha$. (See \cite{HL10}*{\textsection 4.6} and the references therein.)

Assume first that $\alpha=0$, so that $\E \in \DC^b(U \times C)$. We apply $L\pi^\ast$ to get $(1 \times L\pi^\ast)(\E)$ in $\DC^-(\Rcal)_U$. By flat base change, we have that over each point $p \in U$, the object $(1 \times L\pi^\ast)(\E)$ restricts to $L\pi^\ast E=\pi^\ast E$, where $E$ is the vector bundle corresponding to $U$. In particular, this shows that $(1 \times L\pi^\ast)(\E) \in \DC^b(\Rcal)_U$. Moreover, by Lemma \ref{lemma:pfcomparison_pullback}, we have that this is a family of $\sigma$-semistable objects. This way, the family $(1 \times L\pi^\ast) \E$ defines a map $U \to \M_\sigma(v)$, and so a map $s\colon U \to M_\sigma(v)$ by composition. By construction, we get that the induced composition $U \to M_\sigma(v) \to M_C(r, \overline{d})$ is the identity on $U$, and so $s$ is a section of $R\pi_\ast\colon M_\sigma(v) \to M_C(r, \overline{d})$ over $U$. 

To conclude, we need to show that $(R\pi_\ast)^{-1}(U) = s(U)$; i.e., that there are no other objects in $M_\sigma(v)$ mapping to $U$. We argue by contradiction: denote by $\F \in M_\sigma(v)$ an object satisfying $f\colon E \cong \pi_\ast \F \in M_C(r, \overline{d})$. By adjunction, the map $f$ induces a map $f'\colon \pi^\ast E \to \F$. By assumption, we have that $v(\pi^\ast E)=v(\F)$.

We claim that $f'$ is an isomorphism. To show this, let us write $\pi^\ast E = (G, \overline{G}; \psi)$ and $\F=(F, \overline{F}; \phi)$. The map $G \to F$ is generically an isomorphism, as $f$ is an isomorphism. As $G$ is torsion-free, and that $G, F$ have the same degree and rank, we conclude that the map $G \to F$ is an isomorphism. The surjectivity of $\psi$ ensures that the kernel of $\overline{G} \to \overline{F}$ is trivial; thus, it is also an isomorphism as $v(\overline{G})=v(\overline{F})$. We get that $f'$ is an isomorphism, as required. This finishes the proof of Theorem \ref{teo:comparison_main}, under the assumption that $\alpha=0$.

For the general case, pick an étale cover $\{V_i \to U\}$ representing $\alpha$, so that $\E|_{V_i} \in \DC^b(V_i \times C)$. By construction, the restrictions of $\E|_{V_i}$ and $\E|_{V_j}$ to $V_{ij} = V_i \times_U V_j$ are isomorphic, and the isomorphism is induced by an element $\lambda_{ij} \in \Gamma(V_{ij}, \O^\times)$. The previous constructions gives étale-local sections $V_i \to \M_\sigma(v)$ that differ by $\lambda_{ij}$, and so the induced sections $V_i \to M_\sigma(v)$ agree in the intersections. By descent, we get an honest section $U \to M_\sigma(v)$. 
\end{proof}

\section{Nodes and cusps} \label{sec:node}

In this section we will specialize to the case when $C$ is an irreducible curve with a single node or a cusp. We modify the notation of Subsection \ref{subsec:setup} accordingly: denote $q \in C$ the singular point of the curve, and $S \subset \tilde{C}$ its set-theoretic preimage, which is a subscheme of length 2. We also let $g$ be the genus of $\tilde{C}$, so that $C$ has arithmetic genus $g+1$. For the sake of simplicity we will assume $g \geq 1$; see Remark \ref{remark:noderktwo_genuszero} for some comments about $g=0$. 

Let us start by specializing Theorem \ref{teo:intro_main}. In this case, we have seen in Proposition \ref{prop:nrsmall_classification} that $T=\{q\}$, endowed with its reduced structure, is a non-rational locus for $C$. This way, the abelian category $\Coh(\Rcal)$ from Subsection \ref{subsec:abcat} can be written as 
\[ \Coh(\Rcal) = \{ (F, \overline{F}; \phi) : F \in \Coh(\tilde{C}), \overline{F} \in \Coh(q), \phi\colon \rho_\ast F|_S \to \overline{F} \}. \]
The functors of Subsection \ref{subsec:functors} specialize to
\begin{gather*}
a_\ast(F, \overline{F}; \phi) = F, \qquad a^\ast(M) = (M, \rho_\ast(M|_S); \id); \\
b_\ast(F, \overline{F}; \phi) = \overline{F}, \quad b^\ast(N) = (0, N; 0).
\end{gather*}
Similarly, the two functors of Subsection \ref{subsec:catres} can be written as 
\begin{gather*}
\pi_\ast(F, \overline{F}; \phi) = \ker(\rho_\ast F \to \rho_\ast (F|_S)\xrightarrow{\phi} \overline{F}), \\
\pi^\ast(M) = (\rho^\ast M, \coker(M|_T \to (\rho_\ast \rho^\ast M|_T)); \text{res}).
\end{gather*}

This way, the category $\DC^b(\Rcal)$ has $K^\num(\Rcal) \cong \Z^3$, induced by the maps $\deg(-)$, $\rk(-)$, $\ell(-)$ from Subsection \ref{subsec:Kgrp}. We use them to identify $v \in K^\num(\Rcal)$ with a triple $v=(r, d, \ell)$.

We constructed in Subsection \ref{subsec:heart} a heart $\Cheart \subset \DC^b(\Rcal)$. An object $E \in \DC^b(\Rcal)$ lies in $\Cheart$ if $H^{-1}(E) = (0, \O_q^{\oplus n}; 0)$, $H^0(E)=(F, \overline{F}; \phi)$ with $\phi$ surjective, and $H^i(E)=0$ otherwise. Lastly, we proved in Corollary \ref{cor:quad_stab} that the central charge
\[ Z(E) = \delta \ell(E) - \gamma \deg(E) + \beta \rk(E) + i \alpha \rk(E) \]
defines a stability condition on $\Cheart$, for $\alpha, \beta, \gamma, \delta$ real numbers subject to the restrictions $\alpha>0$, $\gamma>\delta>0$. We will focus on the case $\alpha=1$, $\beta=0$, $\gamma=1$, and $\delta=t$, for $t \in (0, 1)$. This defines a path $(\sigma_t)_{t \in (0, 1)}$ in the stability manifold $\Stab(\Rcal)$.

\begin{teo} \label{teo:node_main}
Let $r>0$ and $d, \ell \in \Z$ be given, and assume that $v=(r, d, \ell)$ is primitive. For each $t \in (0, 1)$, we let $\M_t(v)=\M_{\sigma_t}(v)$ be the moduli stack of $\sigma_t$ semistable objects with numerical class $v$, and let $\M_t(v) \to M_t(v)$ be the corresponding good moduli space.
\begin{enumerate}
\item For all but finitely many values of $t$, the moduli space $M_t(v)$ parametrizes only stable objects.
\end{enumerate}
Denote by $0<t_1<\dots<t_N<1$ the finitely many values of $t$ for which $M_t(v)$ parametrizes some strictly semistable objects. 
\begin{enumerate}[resume]
\item If $t_i < t, t' < t_{i+1}$, the two moduli spaces $M_t(v), M_{t'}(v)$ are isomorphic.
\item For all $t, t'$, the two moduli spaces $M_t(v)$, $M_{t'}(v)$ are birational.  
\item For $t \neq t_i$, the moduli space $M_t(v)$ is smooth, projective of dimension $r^2(g-1)+1+\ell(2r-\ell)$, and the map $\M_t(v) \to M_t(v)$ is a $\mathbb{G}_m$-gerbe.
\item For $0<t<t_1$, the functor $a_\ast$ induces a morphism $M_t(v) \to M_{\tilde{C}}(r, d)$. If $\gcd(r, d)=1$, the map $M_t(v) \to M_{\tilde{C}}(r, d)$ is a $\Gr(2r, \ell)$-bundle. (In particular, $M_t(v)$ is irreducible for all $t$.)
\item For $t_N <t<1$, the functor $R\pi_\ast$ induces a morphism $M_t(v) \to M_C(r, d+r-\ell)$. If $\ell= r$, then the map is birational.
\end{enumerate}
\end{teo}

\begin{proof}
Parts (1) and (2) are Lemma \ref{lemma:nodewalls_finiteness} and Corollary \ref{cor:nodewalls_samestack}. (The projectivity follows by \cite{BM23}*{Theorem 2.4}, which applies as $M_t(v)$ is smooth.) Part (3) is Proposition \ref{prop:nodewalls_bir}. For part (4), the smoothness is Corollary \ref{cor:nodebasic_smooth}, the dimension is a consequence of Lemma \ref{lemma:nodebasic_main}, while the $\mathbb{G}_m$-gerbe follows from (1). Part (5) is proven in Subsection \ref{subsec:nodecompzero}, and part (6) follows from the discussion in Subsection \ref{subsec:nodecompone}.
\end{proof}

\subsection{Basic properties} \label{subsec:nodebasic}

In this subsection we will give some basic properties of objects in $\Cheart$ corresponding to $\sigma_t$-semistable objects of phase smaller than 1, for some $t$. By Lemma \ref{lemma:sstable_phasenot1}, we have that such objects are of the form $E=(F, \overline{F}; \phi)$, with $F$ locally free.

\begin{lemma} \label{lemma:nodebasic_main}
Let $E=(F, \overline{F}; \phi)$ and $E'=(F', \overline{F}'; \phi')$ be two objects in $\Cheart$, with $F, F'$ locally free. Denote $r=\rk(E)$, $d=\deg(E)$ and $\ell=\ell(E)$, and similarly $r', d', \ell'$. We have the following properties.

\begin{enumerate}
\item There is a distinguished triangle $La^\ast F \to E \to b^\ast \O_q^{\oplus 2r-\ell}[1]$.
\item We have $\Hom^i(E, E')=0$ unless $i=0, 1$. 
\item We have $\chi_\Rcal(E, E') = d'r-dr'+rr'-rr'g-2r\ell' +\ell \ell'$.
\item Assume that $\Hom(E, E)=\C$. Then $\dim \Hom^1(E, E) \leq r^2g +1$, with equality only if $\ell = r$. 
\end{enumerate}
\end{lemma}

\begin{proof}
\begin{enumerate}
\item There is a canonical map $La^\ast F \to E$. Using that $\phi$ is surjective, it follows that $La^\ast F \to E$ is a surjection in $\Coh(\A)$, and its kernel is isomorphic to $(0, \O_q^{\oplus 2r-\ell}; 0)$. The result follows immediately.

\item We take the long exact sequences induced by the triangle in the first part. The result follows immediately by noting the following:
\begin{itemize}
\item $\Hom^i_{\tilde{C}}(F, F')=0$ for $i \neq 0, 1$, as $\tilde{C}$ is smooth.
\item $\Hom^i(La^\ast F, b^\ast \O_q)=0$ for all $i$, by \eqref{eq:sod_main}.
\item $\Hom^i(b^\ast \O_q, La^\ast F)=0$ for $i \neq 0$, by adjunction.
\item $\Hom^i_Y(\O_q, \O_q)=0$ for $i \neq 0$. 
\end{itemize}

\item Follows directly from the (1) and by Riemann--Roch on $\tilde{C}$.

\item From (3) we get that $\dim \Hom^1(E, E) = r^2(g-1) +1 + \ell(2r-\ell)$. Here $0 \leq \ell \leq 2r$, from where the result follows immediately. \qedhere
\end{enumerate}
\end{proof}

A direct application of Lemma \ref{lemma:sstable_phasenot1} ensures that $E$ satisfies the condition of Lemma \ref{lemma:nodebasic_main} if $E$ is $\sigma_t$-semistable of phase smaller than 1 for some $t$.

\begin{cor} \label{cor:nodebasic_smooth}
Let $v=(r, d, \ell)$ be a numerical class. The moduli stack $\M_t(v)$ is smooth. Moreover, if $E$ is stable, then the good moduli space $M_t(v)$ is smooth at the point corresponding to $E$. 
\end{cor}

\begin{proof}
If $E \in \Cheart$ is semistable with numerical class $v$, then Lemma \ref{lemma:nodebasic_main} ensures that $\Hom^2(E, E)=0$. But by \cite{Lie06}, the deformation functor $\Def_E$ has an obstruction theory with values in $\Hom^2(E, E)$. The two results follow immediately. 
\end{proof}

\subsection{Walls} \label{subsec:nodewalls}

The goal of this subsection is to discuss the following result.

\begin{lemma} \label{lemma:nodewalls_finiteness}
Let $(r, d, \ell) \in K^\num(\Rcal)$ be a numerical vector, with $r>0$ and $\gcd(r, d, \ell)=1$. There are finitely many values of $t$, say $0<t_1<\dots<t_N <1$, for which the moduli stack $\M_t(v)$ parametrizes some strictly semistable objects. 
\end{lemma}

Note that, as written, this is a direct consequence of Theorem \ref{teo:comparison_main}. For applications however, it is useful to give an explicit formula for the possible walls.

To do so, fix $r, d, \ell$ as above. Let $E=(F, \overline{F}; \phi)$ be an object of $\Cheart$, with $F$ locally free, with $\rk(E)=r$, $\deg(E)=d$, $\ell(E)=\ell$. Assume that for some $t \in (0, 1)$, there is a proper subobject $E'=(F', \overline{F}'; \phi')$ of $E$ (in $\Cheart$) such that the $\sigma_t$-phase of $E$ and $E'$ agree. The condition that $E$ and $E'$ have the same phase is:
\begin{align}
\frac{-\Im Z_t(E)}{\Re Z_t(E)} = \frac{-\Im Z_t(E')}{\Re Z_t(E')} &\iff \frac{-t \ell +d -\beta r}{r} = \frac{-t \ell'+d'-\beta r'}{r'} \nonumber \\
&\iff t\left( \frac{\ell'}{r'} - \frac{\ell}{r} \right) = \frac{d'}{r'} - \frac{d}{r}. \label{eq:nodewalls_formula}
\end{align}

Note from Theorem \ref{teo:comparison_main} that there are finitely many walls. Thus, there is at most one value of $t$ for which \eqref{eq:nodewalls_formula} holds, namely $t=(d'r-dr')/(\ell'r-\ell r')$. This can be used to compute the walls explicitly in examples. We will go back to this idea in Subsection \ref{subsec:noderktwo}.

\begin{cor} \label{cor:nodewalls_samestack}
Under the previous assumptions, the moduli stacks $\M_t(v)$, $\M_{t'}(v)$ agree if $t_i < t, t' < t_{i+1}$.
\end{cor}

\begin{prop} \label{prop:nodewalls_bir}
Let $(r, d, \ell) \in K^\num(\Rcal)$ be a numerical vector satisfying $r>0$ and $\gcd(r, d, \ell)=1$. For any $t, t' \in (0, 1)$, we have that $M_t(v)$ and $M_{t'}(v)$ are birational.
\end{prop}

\begin{proof}
It suffices to show that for any $t \in (t_i, t_{i+1})$, the maps $M_t(v) \to M_{t_i}(v)$, $M_{t_{i+1}}(v)$ are birational. To do so, we will bound compute the dimension of the strictly semistable loci. 

First of all, note that there are finitely many decompositions $v=u+w$ corresponding to the walls passing through $\sigma_{t_i}$. This way, it suffices to bound the dimension of the locus of $M_t(v)$ that becomes strictly semistable at $\sigma_{t_i}$ with a fixed decomposition $u=v+w$. Denote by $Z \subseteq M_t(v)$ such locus.

Let $E$ be a $\sigma_t$-stable object, that becomes strictly semistable at $\sigma_{t_i}$, fitting in an extension $0 \to A \to E \to B \to 0$. Assume that $A$ and $B$ have all the same phase with respect to $\sigma_{t_i}$, and write $u=[A]=(r', d', \ell')$, $w=[B]=(r'', d'', \ell'')$. 

If $\eta \in T_{Z, [E]} \subseteq \Ext^1(E, E)$ is a tangent vector on $Z$, we have that $\eta$ can be extended into a diagram
\[ \begin{tikzcd}
A \arrow[r] \arrow[d, dashed] & E \arrow[r] \arrow[d, "\eta"] & B \arrow[r] \arrow[d, dashed] & {A[1]} \arrow[d, dashed] \\ {A[1]} \arrow[r] & {E[1]} \arrow[r] & {B[1]} \arrow[r] & {A[2].}
\end{tikzcd} \]
This way, the induced map $A \to B[1]$ must be zero. In other words, we must have that $\eta$ is in the kernel of the map $\Theta\colon \Ext^1(E, E) \to \Ext^1(A, B)$ obtained by composing with the two maps $A \to E$ and $E \to B$. 

Note that $\Theta$ is surjective, as $\Ext^2(E, A)=\Ext^2(B, B)=0$ by Lemma \ref{lemma:nodebasic_main}. This way, it suffices to show that $\dim \Ext^1(A, B)>0$. We can compute this dimension using that $\Hom(A, B)=0$ and Lemma \ref{lemma:nodebasic_main} once again, yielding
\[ \dim \Ext^1(A, B) = r'r''(g-1) + d'r''-d''r'+\ell''(2r'-\ell'). \]
Now, we use that $A, B, E$ have all the same phase at $t_i$, so that$(-t_i \ell'+d')/r' = (-t_i \ell''+d'')/r''=(-t\ell+d)/r$. We get
\begin{align*}
\dim \Ext^1(A, B) &= r'r''(g-1) + t_ir'\ell''-t_i\ell''r' + 2r'\ell''-\ell'\ell'' \\
&\geq t_ir''\ell'-t_i\ell''r' + 2r'\ell''-\ell'\ell'',
\end{align*}
where we used that $g \geq 1$. An elementary argument shows that the right hand side is greater than zero, unless $\ell'=\ell''=0$ or $\ell'=2r'$, $\ell''=2r''$. These cases are excluded by \eqref{eq:nodewalls_formula}, as otherwise $v$ will not be primitive. This shows that $\dim \Ext^1(A, B)>0$, hence the destabilized locus has non-zero codimension.
\end{proof}

\subsection{Comparison at zero} \label{subsec:nodecompzero}

Fix $v=(r, d, \ell)$ with $r>0$ and $\gcd(r, d, \ell)=1$, and fix $0 < t <t_1$, where $t_1$ is the first $v$-wall (as described in Subsection \ref{subsec:nodewalls}). The goal of this subsection is to compare the moduli spaces $M_t(v)$ and $M_{\tilde{C}}(r, d)$.

We proved in Section \ref{sec:comparison} that the functor $a_\ast\colon \DC^b(\Rcal) \to \DC^b(\tilde{C})$ induces a map $\M_t(v) \to \M_{\tilde{C}}(r, d)$ for $t \ll 1$. By the wall-and-chamber decomposition, it suffices to assume that $t<t_1$.  We get an induced map $M_t(v) \to M_{\tilde{C}}(r, d)$ between the corresponding good moduli spaces. Note that the morphism is proper, as $M_t(v)$ and $M_{\tilde{C}}(r, d)$ are proper.

\begin{lemma} \label{lemma:nodecompzero_grass}
Assume that $\gcd(r, d)=1$. We have that $M_t(v) \to M_{\tilde{C}}(r, d)$ is a $\Gr(2r, \ell)$-bundle for the Zariski topology. 
\end{lemma}

\begin{proof}
Let $F$ be a slope-stable vector bundle on $\tilde{C}$ of rank $r$ and degree $d$. Given any surjection $\phi\colon \rho_\ast(F|_S) \to \O_q^{\oplus \ell}$, we claim that the corresponding object $E =(F, \overline{F}; \phi) \in \Cheart$ is $\sigma_t$-stable for $t <t_1$. The proof follows the same lines of Lemma \ref{lemma:pfcomparison_pushforward}. 

This way, we get that for any $F \in M_{\tilde{C}}(r, d)$, the fiber over $F$ is isomorphic to the collection of surjections $\rho_\ast(F|_S) \to \O_q^{\oplus \ell}$ up to isomorphism. This agrees with the Grassmannian $\Gr(\rho_\ast(F|_S), \ell)$. 

We can globalize the previous description. If $\F \in \Coh(M_{\tilde{C}}(r, d) \times \tilde{C})$ is a universal family\footnote{See \citelist{\cite{Ses82}*{Théoremè 18, p. 22} \cite{HL10}*{Corollary 4.6.6}} for the existence of such family.}, consider the relative Grassmannian $G=\Gr(\rho_\ast(\F|_S), \ell)$ over $M_{\tilde{C}}(r, d)$. We get a universal family $\E$ on $\DC^b(\Rcal)_G$, the base change of $\DC^b(\Rcal)$ over $G$, whose fiber over $\phi\colon \rho_\ast(F|_S) \to \O_q^\ell$ is the object $(F, \O_q^\ell; \phi)$. The previous discussion ensures that the induced map $G \to M_t(v)$ is bijective on closed points, and we directly check that is bijective on tangent spaces. By construction, the composition $G \to M_t(v) \to M_{\tilde{C}}(r, d)$ agrees with the bundle map $G \to M_{\tilde{C}}(r, d)$. 

Finally, note that $M_t(v)$ is smooth, as it parametrizes only stable objects. This way, the map $G \to M_t(v)$ is an isomorphism.
\end{proof}

\begin{remark}
Note that Lemma \ref{lemma:nodecompzero_grass} implies that $M_t(v)$ is irreducible if $\gcd(r, d)=1$. Together with Proposition \ref{prop:nodewalls_bir}, we get $M_t(v)$ are irreducible for all $t \in (0, 1)$. 
\end{remark}

\subsection{Comparison at one} \label{subsec:nodecompone}

Fix $v=(r, d, \ell)$ with $r>0$ and $\gcd(r, d, \ell)=1$. Fix $t_N <t<1$, where $t_N$ is the last $v$-wall (as described in Subsection \ref{subsec:nodewalls}). In this subsection we will compare the moduli spaces $M_t(v)$ and $M_C(r, d+r-\ell)$. 

To start, note that by Theorem \ref{teo:comparison_main}, the derived functor $R\pi_\ast$ defines a morphism $M_t(v) \to M_C(r, d+r-\ell)$. Now, note that if $E$ is a vector bundle on $C$, then $v(\pi^\ast E)=(r, d, r)$. This way, in the case that $r=\ell$, we get that the map $R\pi_\ast\colon (R\pi_\ast)^{-1}(U) \to U$ is an isomorphism, where $U \subset M_C(r, d)$ is the open subset parametrizing vector bundles. 

From here, note that $\dim U = r^2(p_a(C)-1)+1 = \dim M_t(v)$. Using Proposition \ref{prop:nodewalls_bir} and Lemma \ref{lemma:nodecompzero_grass}, we know that $M_t(v)$ is irreducible. This way, the map $R\pi_\ast$ is birational.

\subsection{Rank one} \label{subsec:noderkone}

In this subsection we wish to analyze the moduli spaces $M_t(1, d, \ell)$ for all $d \in \Z$ and $\ell=0, 1, 2$. We will then relate the description for $\ell=1$ with the compactified Jacobian of $C$, as discussed by Oda--Seshadri. 

Let us start with $\ell=0$. For general $t$, the moduli space $M_t(1, d, 0)$ is a smooth, projective variety of dimension $g$, thanks to Theorem \ref{teo:node_main}. For $t \ll 1$, we have an induced map $M_t(v) \to M_{\tilde{C}}(1, d)$, which is an isomorphism by Lemma \ref{lemma:nodecompzero_grass}.

Now, note that there are no walls for $t$, thanks to the explicit description in Subsection \ref{subsec:nodewalls}. In fact, the only (numerical) walls for the vector $(r', d', \ell')$ must have $r' <1$, hence $r'=0$. But this is impossible, as $\Im Z_t(0, d', \ell')=0$. It follows that the moduli space $M_t(1, d, 0)$ is independent of $t$.

Note that the map $M_t(1, d, 0) \to M_C(1, d+1)$, factors through $M_C^0(1, d+1)$, the set of non-vector bundles in $C$. This way, the two maps induced by $a_\ast$ and $R\pi_\ast$ fit in the diagram
\[ M_{\tilde{C}}(1, d) \xleftarrow[\cong]{a_\ast} M_t(1, d, 0) \xrightarrow{R\pi_\ast} M_C^0(1, d+1) \subset M_C(1, d+1). \]

The argument goes almost verbatim for $\ell=2$: there are no walls, and the two maps $a_\ast$, $R\pi_\ast$ are isomorphisms. We get the description:
\[ M_{\tilde{C}}(1, d) \xleftarrow[\cong]{a_\ast} M_t(1, d, 2) \xrightarrow[\cong]{R\pi_\ast} M_C^0(1, d-1) \subset M_C^0(1, d-1). \]

The most interesting case is for $\ell=1$. In this case, the map $M_t(1, d, 1) \to M_{\tilde{C}}(1, d)$ is a $\P^1$-bundle. The map $M_t(1, d, 1) \to M_{\tilde{C}}(1, d)$ is birational, and an isomorphism over $M_{\tilde{C}}^1(1, d)$.
\[ M_{\tilde{C}}(1, d) \xleftarrow[\P^1]{a_\ast} M_t(1, d) \xrightarrow[\text{bir}]{R\pi_\ast} M_C(1, d). \]
This recovers the description in \cite{OS79}*{p. 83}.

\subsection{Rank two} \label{subsec:noderktwo}

Our last goal is to compare the moduli spaces $M_C(2, d)$ and $M_{\tilde{C}}(2, d)$ for $d=2k+1$ odd. To do so, we consider the moduli spaces $M_t(2, d, 2)$ for $t \in (0, 1)$, together with the results of Theorem \ref{teo:node_main}.

Let us start by computing the walls. The discussion of Subsection \ref{subsec:nodewalls} gives us a way to compute the $v$-walls for $v=(2, 2k+1, 2)$. We need to look at vectors $v'=(1, d', \ell')$ for $\ell'=0, 1, 2$. The formula $t=(d'r-dr')/(\ell'r-\ell r')$, together with the fact that $t \in (0, 1)$, shows that the only wall corresponds to the decomposition $v=u+w$ with $u=(1, k, 0)$, $w=(1,k+1,2)$. This wall has $t=1/2$. 

This way, there are three distinct moduli spaces $M_t(v)$, for $t \in (0, 1/2)$, $1/2$, and $(1/2, 1)$. We pick a value of $t$ on each interval, say $t=1/4$ and $t=3/4$ respectively. Here, the two inclusions $\M_{1/4}(v), \M_{3/4}(v) \to \M_{1/2}(v)$ induce (proper) morphisms
\[ M_{1/4}(v) \to M_{1/2}(v) \leftarrow M_{3/4}(v) \]
on the corresponding good moduli spaces. 

Let us describe the destabilized loci. If $E \in M_{1/4}(v)$ is strictly semistable at $t=1/2$, it must fit in a short exact sequence $0 \to U \to E \to W \to 0$ in $\Cheart$, with $U \in M_t(u)$ and $W \in M_t(w)$. Using Lemma \ref{lemma:nodebasic_main}, together with the fact that $\Hom(U, W), \Hom(W, U)=0$, we get that $\Hom^1(B, U)$ is $g$-dimensional. 

Similarly, the objects $E' \in M_{3/4}(v)$ that are strictly semistable at $t=1/2$ fit into a sequence $0 \to W \to E' \to U \to 0$ in $\Cheart$, with $U \in M_t(u)$ and $W \in M_t(w)$. A similar computation shows that $\dim \Hom^1(U, W) = g-2$. 

We can globalize the previous description, as follows.  Pick universal families $\U \in \DC^b(\A)_{M_t(u)}$ and $\W \in \DC^b(\A)_{M_t(w)}$. Thanks to Subsection \ref{subsec:noderkone}, the two moduli spaces $M_t(u)$, $M_t(1, k+1, 0)$ are smooth of dimension $g$. We pullback the two universal families to $\DC^b(\A)_{M_t(u) \times M_t(w)}$, take the relative Hom, and take the pushforward to $\DC^b(M_t(u) \times M_t(w))$. We get $\E = R\rho_\ast R\HOM(\pr_2^\ast \W, \pr_1^\ast \U)$, where $R\rho_\ast\colon \DC^b(\A)_{M_t(u) \times M_t(w)} \to \DC^b(M_t(u) \times M_t(w))$, and $\pr_1, \pr_2$ are induced by the projections from $M_t(u) \times M_t(w)$. 

By cohomology and base change, we have that $\E$ is a locally free sheaf of rank $g$ on $M_t(u) \times M_t(w)$. This way, we take $P$ to be the projectivization of $\E$, which is a $\P^{g-1}$-bundle over $M_t(u) \times M_t(w)$. Moreover, it carries a universal family $\Pfam \in \DC^b(\A)_P$, parametrizing all extensions $0 \to A \to E \to B \to 0$ up to isomorphism. This universal family induces a closed embedding $P \hookrightarrow M_{1/4}(v)$, corresponding to the destabilized loci. We point out that $P$ has dimension $(g-1)+g+g=3g-1$.

The same construction applies for the other extensions: the projectivization of $\E' = R\rho_\ast R\HOM(\pr_1^\ast \U, \pr_2^\ast \W)$ gives us a $\P^{g+1}$-bundle $P' \to M_t(u) \times M_t(w)$ parametrizing extensions $0 \to B \to E' \to A \to 0$. This induces an embedding $P' \hookrightarrow M_{3/4}(v)$, corresponding to the destabilized loci at $t=1/2$. Here $P'$ has dimension $3g+1$. 

Finally, let us compare the two moduli $M_{1/4}(v)$, $M_{3/4}(v)$ with the corresponding $M_{\tilde{C}}(2, d)$, $M_C(2, d)$. The map $a_\ast\colon M_{1/4}(v) \to M_{\tilde{C}}(2, d)$ is a $\Gr(4, 2)$-bundle, while $R\pi_\ast\colon M_{3/4}(v) \to M_C(2, d)$ is birational. We point out that $M_t(v)$ has dimension $4g+1$. Putting all together, we have proven the following description.

\begin{teo} \label{teo:noderktwo_main}
Let $d=2k+1$ be an odd integer, and denote $v=(2, d, 2)$. For $t \neq 1/4$, the moduli spaces $M_t(v)$ only parametrize stable objects. For $t<1/2$, the map $a_\ast\colon M_t(v) \to M_{\tilde{C}}(2, d)$ is a $\Gr(4, 2)$-bundle over $M_{\tilde{C}}(2, d)$. For $t>1/2$, the map $R\pi_\ast\colon M_t(v) \to M_C(2, d)$ is birational, and an isomorphism over $M^0_{\tilde{C}}(2, d)$. 

Lastly, for $t=1/2$ there is a wall corresponding to the decomposition $v=u+w$, where $u=(1,k,0)$ and $w=(1,k+1,2)$. For $t<1/2$, the map $M_t(v) \to M_{1/2}(v)$ contracts a $\P^{g-1}$ bundle $P \to M_{1/2}(u) \times M_{1/2}(w)$ to its base. For $t>1/2$, the map $M_t(v) \to M_{1/2}(v)$ contracts a $\P^{g+1}$ bundle $P' \to M_{1/2}(u) \times M_{1/2}(w)$ to its base. In particular, the maps $M_t(v) \to M_{1/2}(v)$ are birational, and so is the map $M_{1/4}(v) \dashrightarrow M_{3/4}(v)$.
\end{teo}

The moduli spaces described in Theorem \ref{teo:noderktwo_main} fit in the following diagram:
\[ \begin{tikzcd}[column sep=small, row sep=small]
& P \arrow[rd, near start, "\P^{g-1}"] \arrow[dd, hook'] & & P' \arrow[ld, near start, "\P^{g+1}"'] \arrow[dd, hook] \\ & & M_{1/2}(u) \times M_{1/2}(w) \arrow[dd, hook] \\ & M_{1/4}(v) \arrow[rr, <->, dashed, near start, "\text{bir}"] \arrow[rd, "\text{bir}"'] \arrow[ld, "{\Gr(4,2)}"] & & M_{3/4}(v) \arrow[ld, "\text{bir}"] \arrow[rd, "\text{bir}"'] \\ M_{\tilde{C}}(2,d) & & M_{1/2}(v) & & M_C(2,d)
\end{tikzcd} \]

\begin{remark} \label{remark:noderktwo_genuszero}
Let us say a few words about the $g=0$ case. Here, $\M_{1/4}(v) = \varnothing$, due to the fact that $M_{\tilde{C}}(2, d) = \varnothing$. As a consequence, we get that $M_{3/4}(v) \cong P'$ is isomorphic to $\P^1$. We point out that in this case $M_{1/4}(v)$ and $M_{3/4}(v)$ are birational.
\end{remark}

\subsection{Higher rank} \label{subsec:nodehigher}

Let us finish our exposition with a few words on the higher rank moduli spaces. Given $(r, d)$ coprime, we would like to relate $M_{\tilde{C}}(r, d)$ with $M_C(r, d)$ as in Theorem \ref{teo:noderktwo_main}. Our main obstruction is that for $r \geq 3$, the walls might involve destabilizing subobjects corresponding to triples $(r', d', \ell')$ that are not primitive. Thus, the moduli spaces $M_t(r', d', \ell')$ might parametrize strictly semistable objects.

Let us say that a wall $v=u+w$ is \emph{simple} if $u$ and $w$ are primitive. The following two examples show that non-simple walls exist; moreover, it is possible that neither $u$ nor $w$ are primitive. 

\begin{example}
\begin{enumerate}
\item For $v=(4, 1, 4)$, there is a wall at $t=1/4$, corresponding to the decomposition $(2, 1, 4)+2(1, 0, 0)$.
\item For $v=(23, 5, 23)$, there is a wall at $t=4/23$, corresponding to the decomposition $3(3, 1, 5) +2(7, 1, 4)$. 
\end{enumerate}
\end{example}

The second type of phenomenon that can occur is having multiple walls for the same value of $t$. In Table \ref{table:nodehigher_616} we described the walls for $v=(6, 1, 6)$. Note that for $t=1/2$ there are two walls. One could overcome this difficulty by deforming our path $\sigma_t$ inside the stability manifold $\Stab(\Rcal)$. The table is computed using \eqref{eq:nodewalls_formula}. (The last column of the table will be explained after Proposition \ref{prop:nodehigher_simpleflip}.)

\begin{table}[htbp]
\centering
\caption{Walls for $v=(6, 1, 6)$.}
\label{table:nodehigher_616}
\begin{tabular}{|c|c|c|c|}
\hline
$t$ & Decomposition & Simple? & Flipped locus \\ \hline \hline
$1/6$ & $(3, 1, 6) + 3(1, 0, 0)$ & No & N/A \\ \hline
$1/4$ & $(3, 1, 5) + (3, 0, 1)$ & Yes & $\P^{9g-6} \to \ast \leftarrow \P^{9g+12}$ \\ \hline
$1/3$ & $(2, 1, 4) + 2(2, 0, 1)$ & No & N/A \\ \hline
\multirow{2}{*}{$1/2$} & $(3, 1, 4) + (3, 0, 2)$ & Yes & $\P^{9g-3} \to \ast \leftarrow \P^{9g+3}$ \\ \cline{2-4}
& $(3, 2, 6) +(3, -1, 0)$ & Yes & $\P^{9g-1} \to \ast \leftarrow \P^{9g+17}$ \\ \hline
$2/3$ & $2(2,1,3) + (2,-1,0)$ & No & N/A \\ \hline
$3/4$ & $(3, 2, 5)+(3, -1, 1)$ & Yes & $\P^{9g} \to \ast \leftarrow \P^{9g+6}$ \\ \hline
$5/6$ & $3(1, 1, 2) + (3, -2, 0)$ & No & N/A \\ \hline
\end{tabular}
\end{table}

The description of the wall-crossing at $t=1/2$ from Theorem \ref{teo:noderktwo_main} generalizes almost verbatim to simple walls.

\begin{prop} \label{prop:nodehigher_simpleflip}
Let $v=(r, d, \ell)$ be a primitive numerical vector with $r>0$. Assume that for some $t \in (0, 1)$ there is a single simple wall $v=u+w$, with $u=(r_1, d_1, \ell_1)$ and $w=(r_2, d_2, \ell_2)$. Without loss of generality, assume that $\ell_1/r_1<\ell_2/r_2$. Then, the birational map $M_{t+\epsilon}(v) \dashrightarrow M_{t-\epsilon}(v)$ is a standard flip, obtained by replacing a $\P^{b-1}$-bundle over $M_t(u) \times M_t(w)$ with a $\P^{c-1}$-bundle over the same base, where $b=d_1r_2-d_2r_1+r_1r_2(g-1) + 2r_1\ell_2 - \ell_1\ell_2$ and $c=d_2r_1-d_1r_2 + r_1r_2(g-1) + 2r_2\ell_1 - \ell_1\ell_2$.
\end{prop}

\begin{proof}
Let us start by determining the local structure of the good moduli space $M_t(v)$ at a point $p$ corresponding to a strictly semistable object. Write $p = [U \oplus W]$, where $U$, $W$ are $\sigma_t$-stable objects of class $u$, $w$ respectively. By Corollary \ref{cor:nodebasic_smooth}, the moduli stack $\M_t(v)$ is smooth at $p$. Together with \cite{CPZ24}*{Lemma 2.7}, we have
\[ \hat{\O}_{M, p} = \C[[\Ext^1(U \oplus W, U \oplus W)]]^{\Aut(U \oplus W, U \oplus W)}. \]

Let us compute all the terms involved in the previous formula. We have that $\Aut(U \oplus W, U \oplus W) \cong (\C^\ast)^2$, as there are no morphisms between $U$ and $W$. We also have
\[ \Ext^1(U \oplus W, U \oplus W) \cong \Ext^1(U, U) \oplus \Ext^1(W, W) \oplus \Ext^1(U, W) \oplus \Ext^1(W, U). \]
The respective dimensions can be computed with Lemma \ref{lemma:nodebasic_main}, giving us $a_1=r_1^2(g-1)+1+\ell_1(2r_1-\ell_1)$, $a_2=r_2^2(g-1)+1 + \ell_2(2r_2-\ell_2)$, $b=d_1r_2-d_2r_1+r_1r_2(g-1) + 2r_1\ell_2 - \ell_1\ell_2$, and $c=d_2r_1-d_1r_2 + r_1r_2(g-1) + 2r_2\ell_1 - \ell_1\ell_2$, respectively. Note that the automorphism group acts by weights $(0,0)$, $(0,0)$, $(1,-1)$, $(-1,1)$, respectively.

This way, we can pick bases $\{x_i\}_{1 \leq i \leq a_1+a_2}$, $\{y_j\}_{1 \leq j \leq b}$, $\{z_k\}_{1 \leq k \leq c}$, so that
\begin{align*}
\hat{\O}_{M, p} &\cong \C[[x_i, y_j, z_k]]^{(\C^\ast)^2} \\
&\cong \C[[x_i, y_jz_k: 1 \leq i \leq a_1+a_2, 1 \leq j \leq b, 1 \leq k \leq c]].
\end{align*}
This corresponds to the contraction of the standard flip at $([U], [W]) \in M_t(u) \times M_t(w)$.

From the construction, we get that the fiber of the map $M_{t+\epsilon}(v) \to M_t(v)$ over $p$ is isomorphic to $\P^{b-1}$, while the fiber of $M_{t-\epsilon}(v) \to M_t(v)$ is isomorphic to $\P^{c-1}$. It is also clear that the two maps correspond to the two small resolutions of $\hat{\O}_{M, p}$. 

To conclude, we have that
\[ b-c = 2d_1r_2 - 2d_2r_1 + 2r_1\ell_2-2r_2\ell_1 = 2r_1r_2(1-t) \left( \frac{\ell_2}{r_2} - \frac{\ell_1}{r_1} \right)>0, \]
as $\ell_2/r_2>\ell_1/r_1$ by assumption. This shows that $b>c$, and so the map $M_{t+\epsilon}(v) \dashrightarrow M_{t-\epsilon}(v)$ is the standard flip. 
\end{proof}

In Table \ref{table:nodehigher_616} we have computed the dimensions of the fibers of the flipped locus $M_{t-\epsilon}(v) \to M_t(v) \leftarrow M_{t+\epsilon}(v)$, denoted by $\P^{c-1} \to \ast \leftarrow \P^{b-1}$.

\section{Tacnodes} \label{sec:tac}

In this section we will specialize to the case when $C$ is a projective, a irreducible curve with a single $A_3$ singularity, also known as a \emph{tacnode}. This is, we assume that there is a point $q \in C$ such that $C \setminus \{q\}$ is smooth, and $\hat{\O}_{C, q} \cong \C[[x, y]]/(y^2-x^4)$. We let $\rho\colon \tilde{C} \to C$ be its normalization, and we get $g$ be the genus of $\tilde{C}$. Denote by $p_1, p_2$ the two preimages of $q$.

To start, note that in this case that $\{q\}$ with its reduced structure is \emph{not} a non-rational locus for $\rho$. Instead, we take $T=V(x^2, y)$, the first-order thickening of $q \in C$ along the unique tangent direction of $C$ at $q$. By the description in Table \ref{table:nrlocal_planar}, we have that this is a non-rational locus for $\rho$. We point out that in this case $S=\rho^{-1}(T)$ has length $2$ at each point $p_i$. 

Let us specialize Theorem \ref{teo:intro_main} to our setting. The categorical resolution from Subsection \ref{subsec:abcat} is
\[ \Coh(\Rcal) = \{ (F, \overline{F}; \phi)\colon F \in \Coh(\tilde{C}), \overline{F} \in \Coh(\A_T), \phi\colon \rho_\ast(\mu_S^\ast (F|_S)) \to \overline{F} \}, \]
where $\Coh(\A_T) = \Coh(\A_{T, \m, 2})$ is the Auslander resolution as in Subsection \ref{subsec:ausalg}. The discussion of Subsection \ref{subsec:Kgrp} allows us to identify $K^\num(\Rcal) \cong \Z^4$, via the maps $\rk, \deg, \ell_1, \ell_2$. Thus, we write $v \in K^\num(\Rcal)$ as a tuple $v=(r, d, \ell_1, \ell_2)$. 

Denote by $\Cheart \subset \DC^b(\Rcal)$ the heart from Subsection \ref{subsec:heart}. As a consequence of Corollary \ref{cor:quad_stab} (cf. the discussion at the beginning of Section \ref{sec:comparison}), we get central charges
\[ Z_{\delta_1, \delta_2}(E) = \delta_1 \ell_1(E) + \delta_2 \ell_2(E)-\deg(E) + i \rk(E), \]
for any $\delta_1, \delta_2 >0$ satisfying $\delta_1 +\delta_2<1$.

\begin{teo} \label{teo:tac_main}
Let $v=(r, d, \ell_1, \ell_2)$ be a vector with $r>0$, $\gcd(r, d)=1$, and $0 \leq \ell_1 \leq 2r$, $0 \leq \ell_2-\ell_1 \leq 2r$. Given $\delta_1, \delta_2>0$ with $\delta_1+\delta_2<1$, we denote by $M_{\delta_1, \delta_2}(v) := M_{\sigma_{\delta_1, \delta_2}}(v)$.
\begin{enumerate}
\item There are finitely many $v$-walls. If $\sigma_{\delta_1, \delta_2}$ and $\sigma_{\delta_1', \delta_2'}$ are outside of the walls, then $M_{\delta_1, \delta_2}(v)$ and $M_{\delta_1', \delta_2'}(v)$ are birational. 
\item For $\delta_1, \delta_2 \ll 1$, the map $M_{\delta_1, \delta_2}(v) \to M_{\tilde{C}}(r, d)$ is a bundle, locally trivial in the Zariski topology, with fiber
\[ \{ (V_1, V_2): V_1 \subset \C^{\oplus 2r}, V_2 \subset \C^{\oplus 4r}, V_1 \times \{0\} \subset V_2 \subset \C^{2r} \times V_1, \dim(V_i) = \ell_i \}. \]

\item Assume that $v=(1, d, 1, 2)$. Then, the moduli space $M_{\delta_1, \delta_2}(v)$ is independent of $(\delta_1, \delta_2)$. We have a diagram
\[ M_{\tilde{C}}(1, d) \xleftarrow{a_\ast} M_{\delta_1, \delta_2}(v) \xrightarrow{R\pi_\ast} M_C(1, d), \]
where the map $a_\ast$ is an $\mathbb{F}_2$-bundle, and the map $R\pi_\ast$ is birational. 

\item Assume that $v=(2, d, 2, 4)$ with $d$ odd. Then, there is a diagram
\[ M_{\tilde{C}}(2, d) \leftarrow M_1 \dasharrow M_2 \dasharrow M_3 \dasharrow M_4 \to M_C(2, d), \]
where $M_1 \to M_{\tilde{C}}(2, d)$ is a bundle, locally trivial in the Zariski topology; each $M_i \dashrightarrow M_{i-1}$ is a standard flip; and $M_4 \to M_C(2, d)$ is birational. 
\end{enumerate}
\end{teo}

\begin{proof}
Part (1) is a consequence of Theorem \ref{teo:comparison_main} and Proposition \ref{prop:nodewalls_bir}; part (2) is Lemma \ref{lemma:taccompzero_flag}; part (3) follows from the discussion of Subsection \ref{subsec:tacrkone}; part (4) is proven in Subsection \ref{subsec:tacrktwo}.
\end{proof}

We will start by discussing some general properties of the Auslander resolution of $\C[\epsilon]/\epsilon^2$ in Subsection \ref{subsec:tacaus}. We will use this description to describe basic properties of the semistable objects in Subsection \ref{subsec:tacbasic}. We will then specialize Theorem \ref{teo:comparison_main} to give a formula for the walls in Subsection \ref{subsec:tacwalls}. Finally, we will describe carefully the situation for rank one and two in Subsections \ref{subsec:tacrkone} and \ref{subsec:tacrktwo}, respectively.

\subsection{The Auslander resolution} \label{subsec:tacaus}

In this section we will describe carefully the category of modules over $\A=\A_{T, \m, 2}$, for $T=\C[\epsilon]/\epsilon^2$. We recall from Lemma \ref{lemma:ausalg_quiver} that $\Coh(\A)$ can be identified with the collection of tuples $(N_1, N_2, \alpha, \beta)$, where $N_1, N_2$ are in $\Coh(T)$, $\alpha\colon N_1 \to N_2$, $\beta\colon N_2 \otimes_T I \to N_1$, and these maps fit into the diagram
\begin{equation} \label{eq:tacaus_prequiver}
\begin{tikzcd} & N_1 \otimes_T I \arrow[r, "\alpha \otimes \id"] \arrow[d, "\times"] \arrow[dl] & N_2 \otimes_T I \arrow[d, "\times"] \arrow[dl, "\beta"'] \\ 0 \arrow[r] & N_1 \arrow[r, "\alpha"] & N_2. \end{tikzcd}
\end{equation}

\begin{lemma}[cf. \cite{KL15}*{Example 5.3}]
The category $\Coh(\A)$ is equivalent to the category of triples $(N_1, N_2; \alpha, \beta)$, where $N_1, N_2$ are finite dimensional $\C$-vector spaces, $\alpha\colon N_1 \to N_2$, $\beta\colon N_2 \to N_1$, subject to $\beta \alpha=0$. 
\end{lemma}

\begin{proof}
Let us look at the diagram \eqref{eq:tacaus_prequiver} carefully. Note from the commutativity of the leftmost triangle that the multiplication map $N_1 \otimes_T I \to N_1$ is zero. Thus, $N_1$ can be identified with a coherent module over $\O_q$, i.e. a vector space. 

On the other hand, note that the map $\beta\colon N_2 \otimes_T I \to N_1$ is determined by $\overline{\beta} = \beta(- \otimes \epsilon) \colon N_2 \to N_1$. Here, the commutativity of the rightmost triangle shows that
\[ \epsilon n = \alpha(\beta(n \otimes \epsilon)) = \alpha \overline{\beta}(n), \]
for any $n \in N_2$. Thus, the $T$-module structure is determined by $\beta$ (together with the $\C$-vector space structure). Finally, the middle triangle gives us the relation $\overline{\beta}\alpha=0$. The process can clearly be reversed.
\end{proof}

\begin{claim}
The category $\Coh(\A)$ has global dimension $\leq 2$. In fact, a more general result is shown in \cite{KL15}*{Proposition A.14}.
\end{claim}

Our next goal is to describe the Euler pairing in $\DC^b(\A)$. To do so, we will use the following observation.

\begin{claim} \label{claim:tacaus_projmodules}
Note that the modules
\[ P^1 = (\C, \C; \id, 0), \qquad P^2 = (\C, \C^2; (1, 0)^T, (0, 1)) \]
are projective. In fact, a fast computation shows that $\Hom(P^1, N) \cong N_1$ and $\Hom(P^2, N) \cong N_2$ canonically, where $N=(N_1, N_2; \alpha, \beta)$.

We also note that $P^2$ is an injective module. In fact, one quickly checks that $N_2^\vee \to \Hom(N, P^2)$, $\psi \mapsto (\psi \circ \alpha, (\psi, \psi \circ \alpha \circ \beta))$ is an isomorphism. 
\end{claim}

\begin{prop} \label{prop:tacaus_Euler}
We have that $K^{\num}(\A) \cong \Z^2$, via $N \mapsto (\ell_1(N), \ell_2(N))$, where $\ell_i(N)=\ell(N_i)$. The Euler pairing is given by the formula
\[ \chi(N, N') = 2\ell_1\ell_1' - \ell_1\ell_2' - \ell_2\ell_1' + \ell_2\ell_2', \]
where we denoted by $\ell_i = \ell_i(N)$, $\ell_i'=\ell_i(N)$.  
\end{prop}

\begin{proof}
The first fact is a consequence of the discussion of Subsection \ref{subsec:Kgrp}. For the second one, note that the Euler pairing is determined by its values on a basis of $K^{\num}(\A)$. It is clear that $P^1, P^2$ provide such a basis; the description of Claim \ref{claim:tacaus_projmodules} allows us to compute directly. 
\end{proof}

The next lemmas will be used in future sections.

\begin{lemma} \label{lemma:tacaus_Ext}
Let $n$ be a positive integer and let $K \subset P_2^{\oplus n}$ be a submodule. We have that $\Ext^i(K, -) = 0$ for all $i \geq 2$, and that $\Ext^1(K, P_2)=0$.
\end{lemma}

\begin{proof}
Let $Q=P_2^{\oplus n}/K$. The first part follows directly from the long exact sequence associated to $0 \to K \to P_2^{\oplus n} \to Q \to 0$, together with the fact that $\Ext^i(Q, -)=0$ for $i \geq 3$. For the second claim, note that $\Ext^1(-, P_2)=0$ as $P_2$ is injective. 
\end{proof}

\begin{lemma} \label{lemma:tacaus_subobj}
Let $n$ be a positive integer and let $K \subset P_2^{\oplus n}$ be a submodule. Then $\ell_1=\ell_1(K)$, $\ell_2=\ell_2(K)$ satisfy the inequalities $0 \leq \ell_1 \leq n$, $0 \leq \ell_2-\ell_1 \leq n$. 

Moreover, for any $\ell_1, \ell_2$ satisfying these inequalities, there exists an submodule $K \subset P_2^{\oplus n}$ with $\ell_1=\ell_1(K)$ and $\ell_2=\ell_2(K)$.
\end{lemma}

\begin{proof}
Note that $P_2^{\oplus n} \cong (\C^n, \C^n \oplus \C^n; \id \oplus 0, (0, I))$. This way, if $K=(K_1, K_2; \alpha, \beta)$ is a subobject, we must have the commutative diagrams:
\[ \begin{tikzcd} K_1 \arrow[r, "\alpha"] \arrow[d, hook] & K_2 \arrow[d, hook] & K_2 \arrow[r, "\beta"] \arrow[d, hook] & K_1 \arrow[d, hook] \\ \C^n \arrow[r, hook, "\id \oplus 0"'] & \C^n \oplus \C^n & \C^n \oplus \C^n \arrow[r, "{(0, \id)}"'] & \C^n. \end{tikzcd} \]
Note that $\ell_1 \leq n$, $\ell_2 \leq 2n$ are automatic. The first diagram implies that $\alpha$ is injective, and so $\ell_1 \leq \ell_2$. A simple diagram chasing on the left diagram shows that $\coker \alpha \hookrightarrow \C^n$, and so $\ell_2-\ell_1 \leq n$. 

Now, assume that $\ell_1, \ell_2$ satisfy the given conditions. If $e_1, \dots, e_n$ is the canonical basis of $\C^n$, take $K_1$ to be the span of $e_1, \dots, e_{\ell_1}$, and $K_2$ to be the span of $e_1 \oplus 0, \dots, e_m \oplus 0$, $0 \oplus e_1, \dots, 0 \oplus e_{m'}$, where $m=\max \{\ell_1, \ell_2-\ell_1\}$ and $m' = \min \{\ell_1, \ell_2-\ell_1\}$. 
\end{proof}

\subsection{Basic properties} \label{subsec:tacbasic}

In this subsection we describe some basic properties of the $\sigma$-semistable objects of phase smaller than 1, in the same spirit of Subsection \ref{subsec:nodebasic}.

\begin{lemma} \label{lemma:tacbasic_main}
Let $E=(F, \overline{F}; \phi)$ and $E' = (F', \overline{F}'; \phi')$ be two objects in $\Cheart$, with $F, F'$ locally free. Write $v(E)=(r, d, \ell_1, \ell_2)$ and $v(E')=(r', d', \ell_1', \ell_2')$. 
\begin{enumerate}
\item There is a distinguished triangle $La^\ast F \to E \to b^\ast K[1]$, where $K$ is a submodule of $\rho_\ast(\mu_S^\ast(F|_S))$ with $\ell_1(K)=2r-\ell_1$, $\ell_2(K)=4r-\ell_2$. In particular, we have $0 \leq \ell_1 \leq 2r$ and $0 \leq \ell_2-\ell_1 \leq 2r$. 
\item We have $\Hom^i(E, E')=0$ unless $i=0, 1$. 
\item We have $\chi_\Rcal(E, E') = d'r-dr'+rr'(1-g) -2r\ell_2' + 2\ell_1\ell_1' -\ell_1\ell_2' - \ell_2\ell_1' + \ell_2\ell_2'$.
\end{enumerate}
\end{lemma}

\begin{proof}
We follow the same ideas of Lemma \ref{lemma:nodebasic_main}, with the following changes.
\begin{enumerate}
\item The canonical map $La^\ast F \to E$ is a surjection in $\Coh(\A)$, and the kernel is a submodule of $\rho_\ast \mu_S^\ast(F|_S) \cong P_2^{\oplus 2r}$. 
\item We use the vanishing results proven in Lemma \ref{lemma:tacaus_Ext}. \qedhere
\item We use Proposition \ref{prop:tacaus_Euler}. \qedhere
\end{enumerate}
\end{proof}

\begin{cor} \label{cor:tacbasic_smooth}
Let $v=(r, d, \ell_1, \ell_2)$ be a numerical class. The moduli stack $\M_{\delta_1, \delta_2}(v)$ is smooth. Moreover, if $E$ is stable, then the good moduli space $M_{\delta_1, \delta_2}(v)$ is smooth at the point corresponding to $E$. 
\end{cor}

\begin{proof}
The proof of Corollary \ref{cor:nodebasic_smooth} applies verbatim.
\end{proof}

\subsection{Walls} \label{subsec:tacwalls}

In this section we will give a basic description of the walls corresponding to a primitive vector $v=(r, d, \ell_1, \ell_2)$ with $r>0$, in the same spirit of Subsection \ref{subsec:nodewalls}. To start, note that we have finitely many walls thanks to Theorem \ref{teo:comparison_main}. 

It is useful to describe an explicit formula for the walls. Given $w = (r', d', \ell_1', \ell_2')$, the condition that $v$ and $w$ have the same $Z_{\delta_1, \delta_2}$-slope is
\begin{align} 
\frac{-\delta_1 \ell_1 -\delta_2\ell_2 + d}{r} &= \frac{-\delta_1 \ell_1' -\delta_2\ell_2' + d'}{r'} \nonumber \\ \iff \delta_1 \left( \frac{\ell_1}{r} - \frac{\ell_1'}{r'} \right) + \delta_2 \left( \frac{\ell_2}{r} - \frac{\ell_2'}{r'} \right) &= \frac{d}{r} - \frac{d'}{r'}. \label{eq:tacwalls_formula}
\end{align}

\begin{prop} \label{prop:tacwalls_bir}
Let $v=(r, d, \ell_1, \ell_2) \in K^\num(\Rcal)$ be a primitive numerical vector satisfying $r>0$. Let $\sigma_{\delta_1 \delta_2}$ and $\sigma_{\delta_1'\delta_2'}$ be two stability conditions not contained in any wall. Then $M_{\sigma_{\delta_1\delta_2}}(v)$ and $M_{\sigma_{\delta_1'\delta_2'}}(v)$ are birational.
\end{prop}

\begin{proof}
We follow the same argument of Proposition \ref{prop:nodewalls_bir}. We fix a decomposition $v=u+w$, with $u=(r', d', \ell_1', \ell_2'')$, $w=(r'', d'', \ell_1'', \ell_2'')$. Assume that there is a pair $(\delta_1, \delta_2)$ in the wall, satisfying $0 <\delta_1, \delta_2$, $\delta_1+\delta_2<1$, and $M_{\sigma_{\delta_1\delta_2}}(u), M_{\sigma_{\delta_1\delta_2}}(w)\neq \varnothing$. As in the proof of Proposition \ref{prop:nodewalls_bir}, it suffices to show that $\Ext^1(u, w)=-\chi(u, w)>0$.

To prove this, we use Lemma \ref{lemma:tacbasic_main} to compute $-\chi(u, w)$. Elementary manipulations show that
\begin{align*} \frac{\dim \Ext^1(u, w)}{r'r''} =& (g-1) + \delta_1(A_1-B_1+2) + \delta_2(A_1-B_1+A_2-B_2+4) \\
&\quad -A_1B_1-A_2B_2,
\end{align*}
where $A_1=\ell_1'/r'-2$, $A_2 =(\ell_2'-\ell_1')/2 -2$, $B_1=\ell_1''/r''$, and $B_2=(\ell_2''-\ell_1'')/r''$. Here $-2 \leq A_1, A_2 \leq 0$ and $0 \leq B_1, B_2 \leq 2$. As $\delta_1, \delta_2 >0$, we get that the right hand side is $\geq (g-1)$, and equality holds only if $A_1-B_1+2$, $A_1+A_2-B_1-B_2+4$, $A_1B_1$, $A_2B_2$ are all zero. This implies that $\ell_1', \ell_1', \ell_2', \ell_2''$ are all zero. But this cannot happen thanks to \eqref{eq:tacwalls_formula}.
\end{proof}

\subsection{Comparison at zero}

Fix $v=(r, d, \ell_1, \ell_2)$ a primitive numerical vector. By Theorem \ref{teo:comparison_main}, there exists a morphism $a_\ast\colon M_{\delta_1, \delta_2}(v) \to M_{\tilde{C}}(r, d)$ for $\delta_1, \delta_2 \ll 1$. The goal of this section is to describe this map in the case $\gcd(r, d)=1$.

\begin{lemma} \label{lemma:taccompzero_flag}
Let $v=(r, d, \ell_1, \ell_2)$ be a primitive numerical vector, and assume that $\gcd(r, d)=1$. Then, the map $a_\ast\colon M_\sigma(v) \to M_{\tilde{C}}(r, d)$ is a bundle whose fibers are isomorphic to the flag variety
\[ F = \{ (V_1, V_2): V_1 \subset \C^{\oplus 2r}, V_2 \subset \C^{\oplus 4r}, V_1 \times \{0\} \subset V_2 \subset \C^{2r} \times V_1, \dim(V_i) = \ell_i \}. \]
Moreover, this bundle is locally trivial in the Zariski topology. 
\end{lemma}

\begin{proof}
Let us describe the fiber of a point $F \in M_{\tilde{C}}(1, d)$. This consists of objects $E$ fitting in a short exact sequence $0 \to b^\ast K \to La^\ast F \to E \to 0$, where $K$ has $\ell_1(K) = 2n-\ell_1$, $\ell-2(K) = 4n-\ell_2$. Here, the object $K$ corresponds to the choice of two subspaces $K_1 \subset \C^{\oplus 2n}$, $K_2 \oplus \subset \C^{\oplus 4n}$, satisfying $K_1 \times \{0\} \subset K_2 \subset \C^2 \times K_1$. This gives the description of the proposition on a fiber. The description globalizes in the same way as in Lemma \ref{lemma:nodecompzero_grass}.
\end{proof}

Let us point out that the variety constructed in Lemma \ref{lemma:taccompzero_flag} can be described as follows. Consider the Grassmannian $G=\Gr(r, r-\ell_1)$, and let $\U$ be the tautological subbundle, of rank $\ell_1$. Then there is a forgetful map $F \to G$, mapping $(V_1, V_2)$ to $V_1$. Over each $p=[V_1]$, the fiber is described by an $\ell_2-\ell_1$-dimensional subspace of $\O_G^{2r} \times \U/\U \times \{0\}$; and so $F \cong \Gr(\O_G^{2r} \times \U/\U \times \{0\}, 2r-(\ell_2-\ell_1))$.

\subsection{Rank one} \label{subsec:tacrkone}

In this subsection we consider the moduli spaces of $\sigma$-semistable objects for a numerical vector $v=(1, d, \ell_1, \ell_2)$. As this vector has $r=1$, there are no walls; thus, we simply write $M(v)$. We also have that this moduli space is smooth, thanks to Corollary \ref{cor:tacbasic_smooth}.

Assume that $E \in M(v)$ is a $\sigma$-stable object with vector $v=(1, d, \ell_1, \ell_2)$. We write $E$ fitting into the short exact sequence $0 \to b^\ast K \to La^\ast F \to E \to 0$ in $\Coh(\Rcal)$, where $K$ is an object of $\Coh(\A)$ with $\ell_1(K)=2-\ell_1$, $\ell_2(K)=4-\ell_2$. 

By construction, we have that $K$ is a subobject of $b_\ast La^\ast F \cong P_2^{\oplus 2}$. This way, Lemma \ref{lemma:tacbasic_main} gives us nine possible options for $\ell_1, \ell_2$. We have collected them in Table \ref{table:tacrkone_dim}, together with the value of $\Ext^1(E, E)$ computed by Lemma \ref{lemma:tacbasic_main}.

\begin{table}[htbp]
\centering
\caption{Possible values for $(\ell_1, \ell_2)$.}
\label{table:tacrkone_dim}
\begin{tabular}{|c|ccccccccc|} \hline
$\ell_1$ & 0 & 0 & 0 & 1 & 1 & 1 & 2 & 2 & 2 \\
$\ell_2$ & 0 & 1 & 2 & 1 & 2 & 3 & 2 & 3 & 4 \\  \hline
$\dim \Ext^1(E, E)$ & $g$ & $g+1$ & $g$ & $g+1$ & $g+2$ & $g+1$ & $g$ & $g+1$ & $g$ \\ \hline
\end{tabular}
\end{table}

Note that the description of Lemma \ref{lemma:taccompzero_flag} applies in this situation. We have three situations.
\begin{itemize}
\item If $\dim \Ext^1(E, E)=g$, then the map $M(v) \to M_{\tilde{C}}(1, d)$ is an isomorphism. 
\item If $\dim \Ext^1(E, E)=g+1$, then the map $M(v) \to M_{\tilde{C}}(1, d)$ is a $\P^1$-bundle, locally trivial in the Zariski topology.
\item If $\dim \Ext^1(E, E)=g+2$, then the map $M(v) \to M_{\tilde{C}}(1, d)$ is an $\mathbb{F}_2$-bundle, locally trivial in the Zariski topology, where $\mathbb{F}_2$ is the Hirzebruch surface.  
\end{itemize}

\begin{prop}
Let $v=(1, d, 1, 2)$. The map $R\pi_\ast\colon M(v) \to M_C(1, d)$ from Theorem \ref{teo:comparison_main} is birational.
\end{prop}

\begin{proof}
Note from Theorem \ref{teo:comparison_main} that $R\pi_\ast$ is an isomorphism over the locus of line bundles in $C$. Now, the fact that $a_\ast\colon M(v) \to M_C(1, d)$ is a $\mathbb{F}_2$-bundle implies in particular that $M(v)$ is irreducible. This implies that $R\pi_\ast$ is birational. 
\end{proof}

\subsection{Rank two} \label{subsec:tacrktwo}

In this subsection we will focus ourselves in describing the moduli spaces $M_{\delta_1, \delta_2}(v)$, where $v=(2, 1, 2, 4)$. Note that this is the numerical vector of $L\pi^\ast F$, where $F$ is a vector bundle of rank $2$ and degree $1$ in $\tilde{C}$. We will assume that $g\geq 1$, so that $M_{\tilde{C}}(2, 1) \neq \varnothing$.

We start by computing the walls using \eqref{eq:tacwalls_formula}. We get four walls $W_1, \dots, W_4$, corresponding to the decompositions $v=u+w$ shown in Table \ref{table:tacrktwo_walls}.

\begin{table}[htbp]
\centering
\caption{Walls for $v=(2,1,2,4)$.}
\label{table:tacrktwo_walls}
\begin{tabular}{|c|c|c|c|} \hline
Wall & $u$ & $w$ & Equation \\ \hline \hline
$W_1$ & $(1, -1, 0, 0)$ & $(1, 2,2,4)$ & $2\delta_1+4\delta_2 = 3$ \\ \hline
$W_2$ & $(1,0,0,0)$ & $(1,1,2,4)$ & $2\delta_1+4\delta_2=1$ \\ \hline
$W_3$ & $(1,0,0,1)$ & $(1,1,2,3)$ & $2\delta_1+2\delta_2=1$ \\ \hline
$W_4$ & $(1,0,0,2)$ & $(1,1,2,2)$ & $2\delta_1=1$ \\ \hline
\end{tabular}
\end{table}

These four walls determine five chambers in the region $\{\sigma_{\delta_1, \delta_2}: 0 < \delta_1, \delta_2, \delta_1+\delta_2 <1\}$. We have depicted all four walls in Figure \ref{fig:tacrktwo_wc}.

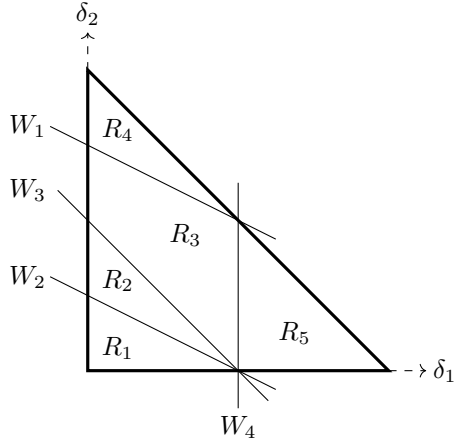
\begin{figure}[htbp]
\centering
\begin{tikzpicture}
\draw[very thick] (0,0) -- (4,0) -- (0,4) -- cycle;
\draw[dashed, ->] (4,0) -- (4.5,0);
\draw[dashed,->] (0,4) -- (0,4.5);
\draw (2,-0.5) -- (2,2.5);
\draw (-0.5,3.25) -- (2.5,1.75);
\draw (-0.4,2.4) -- (2.4,-0.4);
\draw (-0.5,1.25) -- (2.5,-0.25);

\node at (4.75,0) {$\delta_1$};
\node at (0,4.75) {$\delta_2$};
\node at (-0.8,3.25) {$W_1$};
\node at (-0.8,1.25) {$W_2$};
\node at (-0.8,2.4) {$W_3$};
\node at (2,-0.75) {$W_4$};
\node at (0.4,0.3) {$R_1$};
\node at (0.4,1.2) {$R_2$};
\node at (1.3,1.8) {$R_3$};
\node at (0.4,3.2) {$R_4$};
\node at (2.75,0.5) {$R_5$};
\end{tikzpicture}
\caption{Walls and chambers of $v=(2,1,2,4)$.}
\label{fig:tacrktwo_wc}
\end{figure}

For each wall $W_i$ associated to the decomposition $v=u+w$, we need to compute which chambers have $\phi(u)<\phi(w)$, and the dimensions of $\Ext^1(u, w)$, $\Ext^1(w, u)$. We have summarized this information in Table \ref{table:tacrktwo_dimwalls}.

\begin{table}[htbp]
\centering
\caption{Dimensions and chambers for the walls.}
\label{table:tacrktwo_dimwalls}
\begin{tabular}{|c|c|c|c|c|c|} \hline
Wall & $\dim M(u)$ & $\dim M(w)$ & $\dim \Ext^1(u, w)$ & $\dim \Ext^1(w, u)$ & $\phi(u)<\phi(w)$ \\ \hline \hline
$W_1$ & $g$ & $g$ & $g+4$ & $g+2$ & $R_1, R_2, R_3, R_5$ \\ \hline
$W_2$ & $g$ & $g$ & $g+6$ & $g$ & $R_1$ \\ \hline
$W_3$ & $g+1$ & $g+1$ & $g+3$ & $g+1$ & $R_1, R_2$ \\ \hline
$W_4$ & $g$ & $g$ & $g+2$ & $g+4$ & $R_1, R_2, R_3, R_4$ \\ \hline
\end{tabular}
\end{table}

\begin{teo}
For each $1 \leq i \leq 5$, set $M_i=M_{\sigma_i}(v)$, where $\sigma_i =\sigma_{\delta_1, \delta_2}$ for some $(\delta_1, \delta_2) \in R_i$. 
\begin{enumerate}
\item Each of the moduli spaces $M_i$ is smooth, projective, irreducible, of pure dimension $4g+5$. 
\item The map $a_\ast\colon M_1 \to M_{\tilde{C}}(2, 1)$ from Theorem \ref{teo:comparison_main} is an bundle whose fibers are isomorphic to the variety
\[ F=\{ (V_1, V_2): V_1 \subset \C^{\oplus 2}, V_2 \subset \C^{\oplus 4}, V_1 \times \{0\} \subset V_2 \subset \C^2 \times V_1, \dim(V_i) = \ell_i \}. \]
The bundle $a_\ast$ is locally trivial in the Zariski topology.
\item Given $R_i, R_j$ two adjacent chambers, we let $M_{ij} = M_{\sigma_{\delta_1, \delta_2}}(v)$ for a general point $(\delta_1, \delta_2) \in \overline{R_i} \cap \overline{R_j}$. Then, the natural maps $M_i \to M_{ij} \leftarrow M_j$ are a standard flip diagram. 
\item The moduli spaces $M_1, \dots, M_5$ are birational. 
\item The morphism $R\pi_\ast\colon M_4 \to M_C(2, 1)$ from Theorem \ref{teo:comparison_main} is birational.
\end{enumerate}
\end{teo}

\begin{proof}
The proof follows the same ideas of Subsection \ref{subsec:noderktwo}.
\begin{enumerate}
\item Smoothness follows immediately from Corollary \ref{cor:tacbasic_smooth}, while the dimension can be computed using Lemma \ref{lemma:tacbasic_main}. Irreducibility will follow from (2) and (3). 

\item Follows from Lemma \ref{lemma:taccompzero_flag}.

\item The proof of Proposition \ref{prop:nodehigher_simpleflip} applies verbatim.

\item We only need to check that the flipped loci have dimension smaller than $4g+5$. This follows immediately from the computations in Table \ref{table:tacrktwo_dimwalls}.

\item The proof of the previous point ensures that $M_5$ is irreducible of dimension $4g+5$. By Theorem \ref{teo:comparison_main}, we know that a $4g+5$-dimensional open subset of $M_5$ is isomorphic to an open subset of $M_C(2, d)$. \qedhere
\end{enumerate}
\end{proof}

\begin{cor}
The moduli space $M_C(2, 1)$ is birational to an $F$-bundle over $M_{\tilde{C}}(2, 1)$. More explicitly, we have a factorization
\[ M_{\tilde{C}}(2, 1) \leftarrow M_1 \dasharrow M_2 \dasharrow M_3 \dasharrow M_4 \to M_C(2, 1), \]
described as follows:
\begin{enumerate}
\item The map $M_1 \to M_{\tilde{C}}(2, 1)$ is an $F$-bundle, locally trivial in the Zariski topology. 
\item The map $M_1 \dashrightarrow M_2$ is a standard antiflip, replacing a $\P^{g-1}$-bundle over $M(u_2) \times M(w_2)$ with a $\P^{g+5}$-bundle. 
\item The map $M_2 \dashrightarrow M_3$ is a standard antiflip, replacing a $\P^g$-bundle over $M(u_3) \times M(v_3)$ with a $\P^{g+2}$-bundle.
\item The map $M_3 \dashrightarrow M_4$ is a standard antiflip, replacing a $\P^{g+1}$-bundle over $M(u_1) \times M(w_1)$ with a $\P^{g+3}$-bundle. 
\item The map $M_4 \to M_C(2, 1)$ is birational.
\end{enumerate}
\end{cor}

\bibliography{Curves.bbl}
\end{document}